\documentclass[11pt]{amsart}
\usepackage{amssymb,mathrsfs,graphicx,enumerate}
\usepackage{amsmath,amsfonts,amssymb,amscd,amsthm,bbm}
\usepackage{mathtools}
\mathtoolsset{showonlyrefs}
\usepackage[retainorgcmds]{IEEEtrantools}
\usepackage{colortbl}
\usepackage{lipsum}
\usepackage{graphicx, subfigure}
\usepackage[caption=false]{subfig}
\usepackage{graphicx}
\title[Interaction models on manifolds with bounded curvature]{Long-time behaviour of interaction models on Riemannian manifolds with bounded curvature}

\author[Fetecau]{Razvan C. Fetecau}
\address[Razvan C. Fetecau]{\newline Department of Mathematics, Simon Fraser University, 8888 University Dr., Burnaby, BC V5A 1S6, Canada}
\email{van@math.sfu.ca}

\author[Park]{Hansol Park}
\address[Hansol Park]{\newline Department of Mathematics, Simon Fraser University, 8888 University Dr., Burnaby, BC V5A 1S6, Canada}
\email{hansol\_park@sfu.ca}

\newtheorem{theorem}{Theorem}[section]
\newtheorem{lemma}{Lemma}[section]
\newtheorem{corollary}{Corollary}[section]
\newtheorem{proposition}{Proposition}[section]
\newtheorem{remark}{Remark}[section]
\newtheorem{example}{Example}[section]

\newtheorem{definition}{Definition}[section]

\newcommand{\bbr}{\mathbb R}

\newcommand{\calM}{\mathcal{M}}
\newcommand{\calP}{\mathcal{P}}
\newcommand{\calE}{\mathcal{E}}
\newcommand{\calS}{\mathcal{S}}
\def\d{\,\mathrm{d}}

\newcommand{\supp}{\mathrm{supp}(\rho)}

\newcommand{\bv}{\mbox{\boldmath $v$}}

\newcommand{\V}{v}
\newcommand{\K}{K}
\newcommand{\uc}{\mu}
\newcommand{\lc}{\lambda}
\newcommand{\g}{g}
\newcommand{\di}{n}
\newcommand{\dist}{d}
\newcommand{\inj}{\mathrm{inj}}
\newcommand{\secc}{\mathcal{K}}
\newcommand{\rw}{r_w}
\newcommand{\rc}{r_c}
\newcommand{\N}{N}
\newcommand{\T}{\bar{t}}
\newcommand{\re}{\zeta}

\makeatletter
\@namedef{subjclassname@2020}{\textup{2020} Mathematics Subject Classification}
\makeatother
\begin{document}

\date{\today}



\begin{abstract}
We investigate the long-time behaviour of solutions to a nonlocal partial differential equation on smooth Riemannian manifolds of bounded sectional curvature. The equation models self-collective behaviour with intrinsic interactions that are modelled by an interaction potential. We consider attractive interaction potentials and establish sufficient conditions for a consensus state to form asymptotically. In addition, we quantify the approach to consensus, by deriving a convergence rate for the diameter of the solution's support. The analytical results are supported by numerical simulations for the equation set up on the rotation group.
\end{abstract}

\maketitle \centerline{\date}
\section{Introduction}
\label{sect:intro}
In this paper we investigate the long-time behaviour of measure-valued solutions to the following integro-differential equation on a Riemannian manifold $M$: 
\begin{subequations}
\label{eqn:model}
\begin{gather}
\partial_t\rho+\nabla_M\cdot(\rho v)=0, \label{eqn:pde} \\
v=-\nabla_M\K*\rho, \label{eqn:v}
\end{gather}
\end{subequations}
where $K: M\times M \to \bbr$ is an interaction potential, and $\nabla_M \cdot$ and $\nabla_M $ represent the manifold divergence and gradient, respectively. In \eqref{eqn:v} the symbol $\ast$ denotes a measure convolution: for a time-dependent measure $\rho_t$ on $M$ and $x\in M$ we set
\begin{align}\label{B-2}
\K*\rho_t(x)=\int_M\K(x, y)\d\rho_t(y).
\end{align}
 
Equation \eqref{eqn:pde} is in the form of a continuity equation that governs the transport of the measure $\rho$ along the flow on $M$ generated by the velocity field $v$ given by \eqref{eqn:v}. Note that \eqref{eqn:pde} is an {\em active} transport equation, as the velocity field has a nonlocal functional dependence on $\rho$ itself. This geometric interpretation plays a major role in the paper, as measure-valued solutions to \eqref{eqn:model} are defined via optimal mass transport \cite{CanizoCarrilloRosado2011,FePa2021}.  Also, since \eqref{eqn:model} conserves the total mass, we restrict our solutions to be probability measures on $M$ at all times, i.e., $\int_M  \d \rho_t=1$ for all $t \geq 0$. We note that system \eqref{eqn:model} has a discrete/ODE analogue that has an interest in its own \cite{CaChHa2014}, and the general framework of measure-valued solutions includes the discrete formulation as well.

In the literature, equation \eqref{eqn:model} is interpreted as an aggregation or interaction model, with $\rho$ representing the density of a certain population. Indeed, the velocity $v$ at $x$ computed by \eqref{eqn:v} accounts for all contributions from interactions with point masses $y$ in the support of $\rho$, through the convolution. As a result of this interaction, point $x$ moves either toward or away from $y$, depending whether the interaction with $y$ is attractive or repulsive. The nature of the interaction between $x$ and $y$ is set by the vector $-\nabla_M K(x,y)$ (the gradient is taken with respect to $x$), which also provides the direction and magnitude of their interaction \cite{FeZh2019,FePaPa2020,FePa2021}. With such interpretation, model \eqref{eqn:model} has numerous applications in swarming and self-organized behaviour in biology \cite{CarrilloVecil2010}, material science \cite{CaMcVi2006}, robotics \cite{Gazi:Passino,JiEgerstedt2007}, and social sciences \cite{MotschTadmor2014}. Depending on the application, equation \eqref{eqn:model} can model interactions between biological organisms (e.g., insects, birds or cells), robots or even opinions. 

While there is extensive recent analytical and numerical work on solutions to model \eqref{eqn:model}, this research has focused almost exclusively on the model set up on the Euclidean space $\bbr^\di$. For analysis of \eqref{eqn:model} in Euclidean setups we refer to \cite{BodnarVelasquez2, BertozziLaurent, Figalli_etal2011, BeLaRo2011} for the well-posedness of the initial-value problem, to \cite{LeToBe2009, FeRa10, BertozziCarilloLaurent, FeHuKo11,FeHu13} for the long-time behaviour of its solutions, and to \cite{Balague_etalARMA, ChFeTo2015, SiSlTo2015} for studies on minimizers for the associated interaction energy. Numerically, it has been shown that model \eqref{eqn:model} can capture a wide variety of self-collective or swarm behaviours, such as aggregations on disks, annuli, rings and soccer balls \cite{KoSuUmBe2011, Brecht_etal2011, BrechtUminsky2012}.

The literature on solutions to model \eqref{eqn:model} set up on general Riemannian manifolds is very limited, with only a few works on this subject. In this respect we distinguish two approaches: extrinsic and intrinsic. In the {\em extrinsic} approach the manifold $M$ is assumed to have a natural embedding in a larger Euclidean space, and interactions between points on $M$ depend on the Euclidean distance in the ambient space between the points \cite{WuSlepcev2015,CarrilloSlepcevWu2016,PaSl2021}. The second approach considers {\em intrinsic} interactions, which only depend on the intrinsic geometry of the manifold  \cite{FePa2021, FeZh2019}. In particular, the interaction potential can be assumed in this case to depend on the geodesic distance on the manifold between points \cite{FePa2021, FeZh2019}. The goal of the present paper is to consider such interaction potentials and take a fully intrinsic approach to study the long time behaviour of solutions  to \eqref{eqn:model} on general Riemannian manifolds with bounded sectional curvature.  

Well-posedness of measure-valued solutions to equation \eqref{eqn:model} set up on general Riemannian manifolds, with intrinsic interactions, has been established recently in \cite{FePa2021}. Previous works considered the equation on particular manifolds, such as sphere and cylinder \cite{FePaPa2020} and the special orthogonal group $SO(3)$ \cite{FeHaPa2021}. The long-time behaviour of solutions to \eqref{eqn:model} on manifolds has also been considered recently. Rich pattern formation behaviours have been shown for the model with intrinsic interactions on the sphere and the hyperbolic plane \cite{FeZh2019, FePaPa2020}, and on the special orthogonal group $SO(3)$ \cite{FeHaPa2021}. For the extrinsic approach, emergent behaviour has been studied on various manifolds such as sphere \cite{HaChCh2014},  unitary matrices \cite{Lohe2009, HaRy2016, HaKoRy2018}, hyperbolic space \cite{Ha-hyperboloid}, and Stiefel manifolds \cite{HaKaKi2022}. 

In the present research we consider purely attractive potentials and investigate the long time behaviour of the solutions to  \eqref{eqn:model}. For strongly attractive potentials, the focus is on the formation of consensus solutions, where the equilibria consist of an aggregation at a single point. In the engineering literature, achieving such a state is also referred to as synchronization or rendezvous. Bringing a group of agents/robots to a rendezvous configuration is an important problem in robotic control  \cite{Sepulchre2011,TronAfsariVidal2012, Markdahl_etal2018}. We also note here that applications of \eqref{eqn:model} in engineering (robotics) often require a manifold setup, as agents/robots are typically restricted by environment or mobility constraints to remain on a certain manifold. In such applications, for an efficient swarming or flocking, agents must approach each other along geodesics, further motivating the intrinsic approach taken in this paper. The emergence of self-synchronization has also numerous occurrences in biological, physical and chemical systems (e.g., flashing of fireflies, neuronal synchronization in the brain, quantum synchronization) -- see \cite{Lohe2009, Ha-hyperboloid} and references therein. For applications of asymptotic consensus to opinion formation, we refer to \cite{MotschTadmor2014}. 

Emergence of asymptotic consensus in intrinsic interaction models on Riemannian manifolds has been studied recently for certain specific manifolds such as sphere \cite{FePaPa2020} and rotation group \cite{FeHaPa2021}, as well as for general manifolds of constant curvature \cite{FeHaPa2021}. In the current paper we take a very general approach and investigate the formation of consensus on arbitrary manifolds of bounded curvature. We establish sufficient conditions on the interaction potential  and on the support of the initial density for a consensus state to form asymptotically. Compared to the most general results available to date \cite{FeHaPa2021}, the present research improves in several key aspects (see Remark \ref{remark:compare}): i) considers general manifolds of bounded curvature, ii) relaxes the assumptions on the interaction potential, and iii) provides a quantitative rate of convergence to consensus. In particular, by relaxing the assumption on $K$, our study includes now the important class of power-law potentials. We also provide numerical experiments for $M=SO(3)$ and show that the analytical rate of convergence to consensus is sharp.

The paper is structured as follows. 
 In Section \ref{sect:prelim}, we present some preliminaries on the interaction equation \eqref{eqn:model} and on some results from Riemannian geometry; in particular, we set the notion of the solution and the assumptions on the interaction potential. In Section \ref{sect:ss}, we prove the formation of asymptotic consensus for strongly attractive potentials (Theorem \ref{T3.1}).  With an additional assumption on the interaction potential, in Section \ref{sect:conv} we quantify the approach to consensus by establishing a rate of convergence for the diameter of the support of the solution (Theorem \ref{thm:rate-conv}).  In Section \ref{sect:weak-att} we consider weakly attractive potentials and investigate the asymptotic behaviour (Theorem \ref{Convt}).  In Section \ref{sect:numerics}, we present some numerical results for $M=SO(3)$. Finally, the Appendix includes some additional comments on well-posedness and the proofs of several lemmas.



\section{Preliminaries}
\label{sect:prelim}

In this section, we introduce the notion of solutions to model \eqref{eqn:model} and discuss the well-posedness of solutions, as established in \cite{FePa2021}. We then present the gradient flow formulation of model \eqref{eqn:model}, and some concepts and results from Riemannian geometry that we will use in the paper.
 
The following assumptions on the manifold $M$ and interaction potential $K$ will be made throughout the entire paper.  
\smallskip

\noindent(\textbf{M}) $M$ is a complete, smooth Riemannian manifold of finite dimension $\di$, with positive injectivity radius $\inj(M)>0$. We denote its intrinsic distance by $\dist$ and sectional curvature by $\mathcal{K}$. 
\medskip

\noindent(\textbf{K}) The interaction potential $K:M\times M\to\bbr$ has the form
\[
\K(x, y)=\g(\dist(x, y)^2),\qquad \text{ for all }x, y\in M,
\]
where $\g:[0, \infty)\to\bbr$ is differentiable, $\g'$ is locally Lipschitz continuous, and
\begin{equation}
\g'(r^2)\geq 0,\qquad \text{ for all }0< r < \inj(M).
\label{eqn:Katt}
\end{equation}
In particular, \eqref{eqn:Katt} indicates that the potential $K$ is {\em attractive} -- see explanation below.

Anywhere in the paper, $\cdot$ denotes the inner product of two tangent vectors (in the same tangent space) and $\|\cdot\|$ represents the norm of a tangent vector. The tangent bundle of $M$ is denoted by $TM$. We will also omit the subscript $M$ on the manifold gradient and divergence.
 
The expression \eqref{eqn:v} for $v$ (see also \eqref{B-2}) involves the gradient of $K$, which in turn is a function of the squared distance function $d^2$. For all $x,y\in M$ with $\dist(x,y) < \inj(M)$, the gradient with respect to $x$ of $d^2$ is given by
\begin{equation}
\label{eqn:grad-d2}
\nabla_x d(x,y)^2 = -2 \log_x y,
\end{equation}
where $\log_x y$ denotes the Riemannian logarithm map (i.e., the inverse of the Riemannian exponential map) \cite{doCarmo1992}. Hence, by chain rule we have
\begin{equation}
\label{eqn:gradK-gen}
	\nabla_x K(x,y) = -2 g'(d(x,y)^2) \log_x y.
\end{equation}

The velocity at $x$, as computed by \eqref{eqn:v}, considers all interactions with point masses $y \in \supp $.  By \eqref{eqn:gradK-gen}, when a point mass at $x$ interacts with a point mass at $y$ (we assume here $\dist(x,y) < \inj(M)$), the mass at $x$ is driven by a force of magnitude proportional to $|g'(d(x,y))^2|d(x,y)$, to move either towards $y$ (provided $g'(d(x,y)^2) >0$) or away from $y$  (provided $g'(d(x,y)^2) < 0$).  If $g'(x,y)=0$, the two point masses do not interact at all. For a potential that satisfies (\textbf{K}), any two point masses within $\inj(M)$ distance from each other either feel an attractive interaction, or do not interact at all.
 
\subsection{Notion of the solution and well-posedness}
\label{subsect:solution}

Denote by $U \subset M$ a generic open subset of $M$, and by $\calP(U)$ the set of Borel probability measures on the metric space $(U,d)$. Also denote by $C([0,T);\calP(U))$ the set of continuous curves from $[0,T)$ into $\calP(U)$ endowed with the narrow topology (i.e., the topology dual to the space of continuous bounded functions on $U$; see \cite{AGS2005}). 

For $\Psi: \Sigma \to U$, with $\Sigma\subset U$ a measurable set, we denote by $\Psi \# \rho$ the \emph{push-forward} in the mass transport sense of $\rho$ through $\Psi$. Equivalently, $\Psi\#\rho$ is the probability measure such that for every measurable function $\varphi : U \to [-\infty,\infty]$ with $\varphi\circ \Psi$ integrable with respect to $\rho$, it holds that
\begin{equation}
\label{eqn:push-forward}
	\int_U \varphi(x) \d (\Psi \# \rho)(x) = \int_\Sigma \varphi(\Psi(x)) \d\rho(x).
\end{equation}

We define solutions to \eqref{eqn:model} in a geometric way, as the push-forward of an initial density $\rho_0$ through the flow map on $M$ generated by the velocity field $\V$ given by \eqref{eqn:v} \cite[Chapter~8.1]{AGS2005}. To set up terminology, consider a time-dependent vector field $X : U \times [0,T) \to TM$ and a measurable subset $\Sigma \subset U$. The \emph{flow map} generated by $(X,\Sigma)$ is a function $\Psi_X : \Sigma \times [0,\tau) \to U$, for some $\tau\leq T$, that for all $x\in\Sigma$ and $t\in[0,\tau)$ satisfies
\begin{equation} \label{eqn:characteristics-general}
  \begin{cases} \frac{\d}{\d t} \Psi^t_X(x) = X_t(\Psi^t_X(x)),\\[7pt]
    \Psi^0_X(x) = x, \end{cases}
\end{equation}
where we used the notation $X_t$ to denote $X(\cdot,t)$ and $\Psi^t_X$ for $\Psi_X(\cdot,t)$. A flow map is said to be \emph{maximal} if its time domain cannot be extended while \eqref{eqn:characteristics-general} holds; it is said to be \emph{global} if $\tau=T=\infty$ and \emph{local} otherwise.

For $T>0$ and a curve $(\rho_t)_{t\in [0,T)} \subset \calP(U)$, the vector field in model \eqref{eqn:model} is $\V$ given by \eqref{eqn:v}. To indicate its dependence on $\rho$, let us rewrite \eqref{eqn:v} as
\begin{equation} \label{eqn:v-field}
	\V[\rho] (x,t) =  -\nabla K *\rho_t (x), \qquad \mbox{for $(x,t) \in U \times [0,T)$},
\end{equation}
where for convenience we used $\rho_t$ in place of $\rho(t)$, as we shall often do in the sequel. 

We adopt the following definition of solutions to equation \eqref{eqn:model}.
\begin{definition}[Notion of weak solution]\label{defn:sol}
Given $U\subset M$ open, we say that $(\rho_t)_{t\in[0, T)}\subset \calP(U)$ is a weak solution to \eqref{eqn:model} if $(v[\rho], \mathrm{supp}(\rho_0))$ generates a unique flow map $\Psi_{v[\rho]}$ defined on $\mathrm{supp}(\rho_0)\times[0, T)$, and $\rho_t$ satisfies the implicit transport equation
\[
\rho_t=\Psi^t_{v[\rho]}\#\rho_0,\qquad \text{ for all } t\in [0, T).
\]
\end{definition}
Note that by Lemma 8.1.6 of \cite{AGS2005}, a weak solution to \eqref{eqn:model} defined as above is also a weak solution in the sense of distributions. The local well-posedness of solutions to model \eqref{eqn:model} (in the sense of Definition \ref{defn:sol}) was established in \cite[Theorem 4.6]{FePa2021}. For purely attractive potentials, the local well-posedness can be upgraded to global well-posedness \cite[Theorem 5.1]{FePa2021}. 

Before we state the well-posedness result, we introduce some notations which will be used in the paper. For any $\uc\in \bbr$, $\calM_\uc$ denotes the collection of Riemannian manifolds with sectional curvature $\secc$ that satisfies $\lc\leq\secc\leq \uc$ for some $\lc\in\bbr$. In other words, $\calM_\uc$ is the set of Riemannian manifolds with bounded sectional curvature, where $\mu$ is an upper bound of $\secc$.  Also, $B_r(p):=\{x\in M: \dist(x, p)<r\}$ is the open ball centred at $p$, of radius $r$, defined for all $p\in M$ and $r>0$.

Denote
\begin{equation}
\label{eqn:rw}
\rw=\min\left\{
\frac{\inj(M)}{2}, \frac{\pi}{2\sqrt{\mu}}
\right\},
\end{equation}
with the convention that $\frac{1}{\sqrt{\mu}}=\infty$ when $\mu \leq 0$. Note that if $M$ is simply connected, in addition to satisfying (\textbf{M}), then $\inj(M) = \infty$ when $\mu \leq 0$ (cf. \cite[Corollary 6.9.1]{Jost2017}, a consequence of the Cartan--Hadamard theorem). Consequently, in such case $\rw = \infty$. Another remark is that by \cite[Theorem IX.6.1]{Chavel2006}, $B_{\rw}(p)$ is strongly convex for any $p\in M$. In particular, for any two points $x,y\in B_{\rw}(p)$ there exists a unique length-minimizing geodesic connecting $x$ and $y$, that is entirely contained in $B_{\rw}(p)$.

\begin{theorem} (Global well-posedness \cite[Theorem 5.1]{FePa2021})
\label{thm:wp}
Let $M\in \calM_\uc$ for some $\uc\in\bbr$ and $K$ satisfy (\textbf{M}) and (\textbf{K}), respectively. Take an initial density $\rho_0\in \calP(U)$ and $0<r<\rw$ be such that $\mathrm{supp}(\rho_0)\subset \overline{B_r(p)}\subset U$ for some open set $U$. Then, there exists a unique weak solution $\rho$ in $C([0,\infty);\calP(U))$ starting from $\rho_0$ to the interaction equation \eqref{eqn:model}; furthermore, $\mathrm{supp}(\rho_t)\subset\overline{B_r(p)}$ for all $t\geq0$.
\end{theorem}

We note that due to the attractive nature of the potential, $\overline{B_r(p)}$ in Theorem \ref{thm:wp} is an {\em invariant set} for the dynamics. Also, \cite[Theorem 5.1]{FePa2021} does not have $\rw$ as the maximal radius for well-posedness, but lists a smaller value instead. We explain in Appendix \ref{appendix:wp} how a small change to the argument used there leads to the well-posedness result in Theorem \ref{thm:wp}.


\subsection{Wasserstein distance and gradient flow formulation}
\label{subsect:g-flow}

We will use the \emph{intrinsic $2$-Wasserstein distance} to investigate the asymptotic behaviour of solutions to \eqref{eqn:model}. For $U\subset M$ open, and $\rho,\sigma \in \calP(U)$, this distance is defined as:
\begin{equation*}
	W_2(\rho,\sigma) = \left( \inf_{\gamma \in \Pi(\rho,\sigma)} \int_{U\times U} \dist(x,y)^2 \d\gamma(x,y) \right)^{1/2},
\end{equation*}
where $\Pi(\rho,\sigma) \subset \calP(U\times U)$ is the set of transport plans between $\rho$ and $\sigma$, i.e., the set of elements in $\calP(U\times U)$ with first and second marginals $\rho$ and $\sigma$, respectively. 

Denote by $\calP_2(U)$ the set of probability measures on $U$ with finite second moment; when  $U$ is bounded we have $\calP_2(U) = \calP(U)$. The space $(\calP_2(U),W_2)$ is a metric space. 

The energy functional associated to model \eqref{eqn:model} is given by
\begin{align}
\label{eqn:energy}
E[\rho]=\frac{1}{2}\iint_{M \times M}\K(x, y)\d\rho(x)\d\rho(y).
\end{align}
Using \eqref{eqn:energy}, one can write $\V[\rho]$ in \eqref{eqn:v-field} as $\V[\rho]=-\nabla \left(\frac{\delta E[\rho]}{\delta\rho}\right)$. Furthermore, system \eqref{eqn:model} is the gradient flow of the energy $E$ on $(\calP_2(M),W_2)$ (cf. \cite{AGS2005}), i.e.,
\begin{equation}
\label{eqn:gflow}
\partial_t \rho_t=\nabla\cdot\left(\rho_t\nabla\frac{\delta E[\rho_t]}{\delta\rho_t}\right)= -\nabla_{W_2}E[\rho_t].
\end{equation}
A direct calculation also leads to the following decay of the energy \cite{FeZh2019}:
\begin{equation}
\label{eqn:energy-decay}
\frac{\d}{\d t}E[\rho_t]=-\int_M \|v[\rho_t](x)\|^2 \d\rho_t(x)\leq 0.
\end{equation}

\begin{lemma}\label{L2.1}
$(\calP_2(\overline{B_r(p)}), W_2)$ is a compact subset in $(\calP_2(M), W_2)$ for all $0<r<\rw$ and $p\in M$.
\end{lemma}
\begin{proof}
See Appendix \ref{appendix:L2} for the proof.
\end{proof}

Denote by $\calS$ the set of critical points of the energy $E$ with respect to metric $W_2$:
\begin{equation*}
\calS=\{\rho\in \calP_2(M):\nabla_{W_2}E[\rho]\equiv0\}.
\end{equation*}
Based on the gradient flow formulation, solutions to \eqref{eqn:model} have the following asymptotic behaviour.
\begin{proposition}\label{prop:lasalle}
Let $M\in \calM_\uc$ for some $\uc\in\bbr$ and $K$ satisfy (\textbf{M}) and (\textbf{K}), respectively. Take  $\rho_0\in \calP(B_{r}(p))$, with $\mathrm{supp}(\rho_0)\subset \overline{B_r(p)}$ and $0<r<\rw$. Then, we have
\[
\lim_{t\to\infty}W_2(\rho_t, \calE_{r, p})=0,
\]
where 
\begin{align}\label{SE}
\calE_{r, p}:=\calS\cap \calP_2(\overline{B_r(p)}).
\end{align}
\end{proposition}
\begin{proof}
This is a direct consequence of the results and considerations above, along with LaSalle's Invariance Principle. Indeed, $\calP_2(\overline{B_r(p)})$ is compact in $(\calP_2(M), W_2)$ (Lemma \ref{L2.1}) and positively invariant with respect to the dynamics of \eqref{eqn:model} (Theorem \ref{thm:wp}). By the gradient flow formulation (see \eqref{eqn:gflow} and \eqref{eqn:energy-decay}) and LaSalle's Invariance Principle on general metric spaces \cite[Theorem 4.2]{walker2013}, it holds that
\[
\lim_{t\to\infty}W_2\left(\rho_t, \calS \cap \calP_2(\overline{B_r(p)})\right)=0.
\]
\end{proof}

Critical points of the energy are steady states of \eqref{eqn:model}, as given by the next lemma.
\begin{lemma}\label{L2.2}
If $\rho\in \calS $, then 
\[
\V[\rho] =0 \qquad \text{ a.e. with respect to } \rho. 
\]
\end{lemma}
\begin{proof}
Let $\rho\in \calE_{r, p}$. Then, we have $\nabla_{W_2}E[\rho]\equiv 0$ and this yields
\[
-\nabla\cdot(\rho \V[\rho])=\nabla\cdot(\rho\nabla(K*\rho))=0.
\]
From the definition of the weak derivative, for any smooth test function $\varphi$, we have
\[
\int \V[\rho](x)\cdot \nabla \varphi(x) \d\rho(x)=0.
\] 
Since $\nabla\varphi(x)$ is an arbitrary smooth function, $v[\rho]$ is zero a.e. with respect to $\rho$.
\end{proof}


\subsection{Some results from Riemannian geometry}
\label{subsect:riemann}
We state below Rauch's comparison theorem, a key tool we use in our proofs. The Rauch comparison theorem allows us to compare lengths of curves on different manifolds.

\begin{theorem}[Rauch comparison theorem, Proposition 2.5 of \cite{doCarmo1992}]\label{RCT}
Let $M$ and $\tilde{M}$ be Riemannian manifolds that satisfy (\textbf{M}) and suppose that for all $p\in M$, $\tilde{p}\in \tilde{M}$, and $\sigma\subset T_pM$, $\tilde{\sigma}\subset T_{\tilde{p}}\tilde{M}$, the sectional curvatures $\secc$ and $\tilde{\secc}$ of $M$ and $\tilde{M}$, respectively, satisfy
\[
\tilde{\secc}_{\tilde{p}}(\tilde{\sigma})\geq \secc_p(\sigma).
\]
Let $p\in M$, $\tilde{p}\in\tilde{M}$ and fix a linear isometry $i:T_pM\to T_{\tilde{p}}\tilde{M}$. Let $r>0$ be such that the restriction ${\exp_p}_{|B_r(0)}$ is a diffeomorphism and ${\exp_{\tilde{p}}}_{|\tilde{B}_r(0)}$ is non-singular. Let $c:[0, a]\to\exp_p(B_r(0))\subset M$ be a differentiable curve and define a curve $\tilde{c}:[0, a]\to\exp_{\tilde{p}}(\tilde{B}_r(0))\subset\tilde{M}$ by
\[
\tilde{c}(s)=\exp_{\tilde{p}}\circ i\circ\exp^{-1}_p(c(s)),\qquad s\in[0, a].
\]  
Then the length of $c$ is greater or equal than the length of $\tilde{c}$.
\end{theorem}

We will use Theorem \ref{RCT} to compare lengths of curves on $M\in\calM_\uc$ (where $\secc\leq \uc$) and curves on the space of manifolds of constant curvature $\uc$. We recall that on a manifold of constant sectional curvature $\uc \geq 0$, for any points $x, y, z$, the following cosine law holds: \\[5pt]
(Case 1: $\uc>0$) 
\begin{align}\label{B-4}
\begin{aligned}
\cos\left(\sqrt{\uc}\dist(x, y)\right)=& \cos\left(\sqrt{\uc}\dist(x, z)\right) \cos\left(\sqrt{\uc}\dist(y, z)\right)\\
&\hspace{2cm}+\sin\left(\sqrt{\uc}\dist(x, z)\right) \sin\left(\sqrt{\uc}\dist(y, z)\right)\cos\angle(xzy),
\end{aligned}
\end{align}
(Case 2: $\uc=0$)
\begin{align}\label{B-4-1}
\dist(x, y)^2=\dist(x, z)^2+\dist(y, z)^2-2\dist(x, z) \dist(y, z)\cos\angle(xzy).
\end{align}


We combine the Rauch comparison theorem (Theorem \ref{RCT}) and the cosine laws above to obtain the following lemma.
\begin{lemma}\label{L2.3}
Let $M\in \calM_\uc$ with $\uc\in\bbr$ and $x, y, z$ be points on $M$. Then we have the following inequalities for each case.\\
(Case 1: $\uc>0$)
\begin{align}\label{B-4-3}
\begin{aligned}
\cos\left(\sqrt{\uc}\dist(x, y)\right)&\leq \cos\left(\sqrt{\uc}\dist(x, z)\right) \cos\left(\sqrt{\uc}\dist(y, z)\right)\\
&\hspace{2cm}+\sin\left(\sqrt{\uc}\dist(x, z)\right) \sin\left(\sqrt{\uc}\dist(y, z)\right)\cos\angle(xzy),
\end{aligned}
\end{align}
(Case 2: $\uc\leq0$)
\begin{align}\label{B-4-4}
\dist(x, y)^2\geq\dist(x, z)^2+\dist(y, z)^2-2\dist(x, z) \dist(y, z)\cos\angle(xzy).
\end{align}

\end{lemma}
\begin{proof}
(Case 1: $\uc>0$)
Let $\Delta$ be a geodesic triangle with vertices $x, y, z$. We also set three points $\bar{x}, \bar{y}, \bar{z}$ on a manifold $\tilde{M}$ of constant curvature $\uc$, which satisfy 
\begin{align}\label{B-5}
\dist(x, z)=\dist(\bar{x}, \bar{z}),\quad \dist(y, z)=\dist(\bar{y}, \bar{z}),\quad\text{and}\quad\angle(xzy)=\angle(\bar{x}\bar{z}\bar{y}).
\end{align}
We use the cosine law \eqref{B-4} on the manifold $\tilde{M}$ to obtain
\[
\cos\left(\sqrt{\uc}\dist(\bar x, \bar y)\right)= \cos\left(\sqrt{\uc}\dist(\bar x, \bar z)\right) \cos\left(\sqrt{\uc}\dist(\bar y, \bar z)\right)+\sin\left(\sqrt{\uc}\dist(\bar x, \bar z)\right) \sin\left(\sqrt{\uc}\dist(\bar y, \bar z)\right)\cos\angle(\bar x\bar z\bar y).
\]
By substituting \eqref{B-5} into the above relation we find
\begin{align}\label{B-6}
\cos\left(\sqrt{\uc}\dist(\bar x, \bar y)\right)= \cos\left(\sqrt{\uc}\dist(x, z)\right) \cos\left(\sqrt{\uc}\dist(y, z)\right)+\sin\left(\sqrt{\uc}\dist(x, z)\right) \sin\left(\sqrt{\uc}\dist(y, z)\right)\cos\angle(xzy).
\end{align}

In Theorem \ref{RCT}, let the curve $c$ be the length minimizing geodesic which connects $x$ and $y$ on $M$. Then, its image $\tilde{c}$ is a curve that connects $\bar{x}$ and $\bar{y}$ on $\tilde{M}$. By Theorem \ref{RCT},
\begin{equation}
\label{eqn:compare-d}
\dist(x, y)=L(c)\geq L(\tilde{c})\geq \dist(\bar{x}, \bar{y}),
\end{equation}
which yields
\begin{align}\label{B-7}
\cos\left(\sqrt{\uc}\dist(x, y)\right)\leq \cos\left({\uc}\dist(\bar{x}, \bar{y})\right).
\end{align}
Finally, we combine \eqref{B-6} and \eqref{B-7} to obtain the desired result \eqref{B-4-3}.\\

\noindent (Case 2: $\uc\leq0$) Use a very similar idea, and take a comparison triangle that satisfies \eqref{B-5} on a manifold $\tilde{M}$ of constant $0$ curvature. Then, use \eqref{B-4} on $\tilde{M}$ for $\bar x$, $\bar y$ and $\bar z$, and combine with  \eqref{B-5} and \eqref{eqn:compare-d} to obtain \eqref{B-4-4}. We omit the details.
\end{proof}


\section{Strongly attractive potentials: asymptotic behaviour}
\label{sect:ss}
In this section we consider {\em strongly attractive} interaction potentials, which we define as follows. 
\begin{definition}[Strongly attractive potential]\label{defn:strong-pot}
An interaction potential $K$ is called {\em strongly attractive} if it satisfies assumption (\textbf{K}), with \eqref{eqn:Katt} replaced by the stronger condition
\begin{equation}
\g'(r^2) > 0,\qquad \text{ for all }0< r < \inj(M).
\label{eqn:Katt-s}
\end{equation}
\end{definition}
The strict inequality sign in \eqref{eqn:Katt-s} implies that any two points feel a non-trivial attractive interaction. 

We will characterize the set of equilibrium points $\calE_{r, p}$ (see \eqref{SE}), and show asymptotic consensus for strongly attractive potentials. In particular, we will find an explicit form of $\calE_{r, p}$, and establish formation of consensus for initial densities supported in $\overline{B_r(p)}$, for all $p\in M$ and $0<r<\rw$.

\begin{lemma}\label{steady}
Assume that $M\in\calM_\uc$ with $\uc\in\bbr$ satisfies (\textbf{M}) and $K$ is a strongly attractive potential. Then, for all $0<r<\rw$ and $p\in M$, we have
\[
\calE_{r,p}=\{\delta_{q}:q\in \overline{B_r(p)}\}.
\]
\end{lemma}
\begin{proof}
Fix an arbitrary $p \in M$ and $0<r<\rw$. We will show $\calE_{r, p}\subseteq \{\delta_{q}:q\in \overline{B_r(p)}\}$ (first part) and $\calE_{r, p}\supseteq \{\delta_{q}:q\in \overline{B_r(p)}\}$ (second part).\\

\noindent \underline{{\em Part 1}}: $\calE_{r, p}\subseteq \{\delta_{q}:q\in \overline{B_r(p)}\}$. Let $\rho \in \calE_{r, p}$ be an equilibrium density of system \eqref{eqn:model}.  Since $\supp$ is a compact set, define
\[
R:=\max_{x\in \supp}\dist(x, p).
\]
When $R=0$, we define by convention $\overline{B_R(p)}:=\{p\}$. Clearly, $0\leq R\leq r<\rw$.  Take $\tilde{x}\in \supp$ which satisfies $\dist(\tilde{x}, p)=R$, and for any $z\in \overline{B_R(p)}$, consider the geodesic triangle $\Delta\tilde{x}zp$. 
\smallskip

\noindent $\diamond$ (Case 1: $\uc>0$) From Lemma \ref{L2.3}, we have 
\begin{equation}
\label{eqn:cos-ineq}
\cos(\sqrt{\mu}\dist(p, z))\leq \cos(\sqrt{\mu}\dist(p, \tilde{x}))\cos(\sqrt{\mu}\dist(\tilde{x}, z))+\sin(\sqrt{\mu}\dist(p, \tilde{x}))\sin(\sqrt{\mu}\dist(\tilde{x}, z))\cos\angle(p\tilde{x}z).
\end{equation}
Note that 
\[
\dist(\tilde{x}, p)=R < \rw \leq \frac{\pi}{2 \sqrt{\uc}}, \quad \text{ and } \quad \dist(\tilde{x}, z) \leq  \dist(\tilde{x}, p) + \dist(p,z) \leq 2R < \frac{\pi}{\sqrt{\uc}},
\]
and hence, 
\[
\sin(\sqrt{\mu}\dist(p, \tilde{x})) >0, \qquad \sin(\sqrt{\mu}\dist(\tilde{x}, z))>0.
\]

Using \eqref{eqn:cos-ineq} we then get
\[
\frac{\cos(\sqrt{\mu}\dist(p, z))- \cos(\sqrt{\mu}\dist(p, \tilde{x}))\cos(\sqrt{\mu}\dist(\tilde{x}, z))}{\sin(\sqrt{\mu}\dist(p, \tilde{x}))\sin(\sqrt{\mu}\dist(\tilde{x}, z))}\leq \cos\angle(p\tilde{x}z).
\]
Now use the inequality above, the monotonicity of cosine and
\[
\dist(p, z)\leq R = \dist(p, \tilde{x}),
\]
to get
\begin{align*}
\cos\angle(p\tilde{x}z)
&\geq\frac{\cos(\sqrt{\mu}\dist(p, \tilde{x}))- \cos(\sqrt{\mu}\dist(p, \tilde{x}))\cos(\sqrt{\mu}\dist(\tilde{x}, z))}{\sin(\sqrt{\mu}\dist(p, \tilde{x}))\sin(\sqrt{\mu}\dist(\tilde{x}, z))}\\[3pt]
&=\cot(\sqrt{\mu}R)\tan(\sqrt{\mu}\dist(\tilde{x}, z)/2).
\end{align*}

We then obtain
\begin{align*}
\log_{\tilde{x}}p\cdot\log_{\tilde{x}}z&=\dist(\tilde{x}, p)\dist(\tilde{x}, z)\cos\angle(p\tilde{x}z)
\geq\frac{R}{\tan(\sqrt{\mu}R)}\tan(\sqrt{\mu}\dist(\tilde{x}, z)/2)\dist(\tilde{x}, z).
\end{align*}
In the inequality above, use $\sqrt{\mu} \dist(\tilde{x},z)/2 <\pi/2$ and
\[
\tan r\geq r, \qquad \textrm{ for all } r\in[0, \pi/2),
\]
to get
\[
\log_{\tilde{x}}p\cdot\log_{\tilde{x}}z\geq \frac{\sqrt{\mu}R}{2\tan(\sqrt{\mu}R)} \dist({\tilde{x}}, z)^2.
\]
By using \eqref{eqn:gradK-gen}, one can write the velocity field in \eqref{eqn:v-field} as
\begin{equation}
\label{eqn:v-exp}
v[\rho](\tilde{x})=\int_{\supp} 2g'(\dist(\tilde{x}, z)^2)\log_{\tilde{x}}z \d\rho(z).
\end{equation}
We then combine the two equations above to get
\begin{align*}
v[\rho](\tilde{x})\cdot \log_{\tilde{x}}p&=\int_{\supp} 2g'(\dist(\tilde{x}, z)^2)\log_{\tilde{x}}z\cdot\log_{\tilde{x}}p \d\rho(z)\\[5pt]
&\geq  \frac{\sqrt{\mu}R}{\tan(\sqrt{\mu}R)}
\int_{\supp} g'(\dist(\tilde{x}, z)^2)\dist(\tilde{x}, z)^2 \d\rho(z).
\end{align*}

Finally, we find
\begin{equation}
\label{eqn:normv-lb}
\|v[\rho](\tilde{x})\|\geq \frac{1}{R}v[\rho](\tilde{x})\cdot \log_{\tilde{x}}p\geq \frac{\sqrt{\mu}}{\tan(\sqrt{\mu}R)}
\int_{\supp} g'(\dist(\tilde{x}, z)^2)\dist(\tilde{x}, z)^2 \d\rho(z).
\end{equation}
By Lemma \ref{L2.2}, since $\rho$ is an equilibrium density, $v[\rho]$ is zero a.e. with respect to $\rho$. This yields $\|v[\rho](\tilde{x})\|=0$, since $\tilde{x}\in \supp$ and $v[\rho]$ is continuous on $\supp$. Hence, by \eqref{eqn:normv-lb} we have
\begin{equation}
\label{eqn:int-zero}
\int_{\supp} g'(\dist(\tilde{x}, z)^2)\dist(\tilde{x}, z)^2 \d\rho(z)=0.
\end{equation}
Since the integrand above is sign definite, we have
\[
g'(\dist(\tilde{x}, z)^2)d(\tilde{x}, z)^2=0, \qquad \text{ for } z \in \supp, \text{ a.e. with respect to } \rho.
\]
By \eqref{eqn:Katt-s}, this implies that $z=\tilde{x}$ for $z\in \supp$, a.e. with respect to $\rho$. Consequently, $\supp=\{\tilde{x}\}$ and $\rho=\delta_{\tilde{x}}$.
\medskip

\noindent $\diamond$ (Case 2: $\uc\leq 0$) We proceed similarly. From Lemma \ref{L2.3}, we have
\[
\dist(p, z)^2\geq \dist(p, \tilde{x})^2+\dist(\tilde{x}, z)^2-2\dist(p, \tilde{x})\dist(\tilde{x}, z)\cos\angle(p\tilde{x}z).
\]
This yields
\[
\frac{ \dist(p, \tilde{x})^2+\dist(\tilde{x}, z)^2-\dist(p, z)^2}{2\dist(p, \tilde{x})\dist(\tilde{x}, z)}\leq\cos\angle(p\tilde{x}z).
\]
Then use 
\[
\dist(p, z)\leq R = \dist(p,\tilde{x}),
\]
to get
\begin{align*}
\cos\angle(p\tilde{x}z)
&\geq\frac{ \dist(p, \tilde{x})^2+\dist(\tilde{x}, z)^2-\dist(p, \tilde{x})^2}{2\dist(p, \tilde{x})\dist(\tilde{x}, z)}
=\frac{\dist(\tilde{x}, z)}{2R}.
\end{align*}
Hence,
\[
\log_{\tilde{x}}p\cdot\log_{\tilde{x}}z=\dist(\tilde{x}, p)\dist(\tilde{x}, z)\cos\angle(p\tilde{x}z)\geq \frac{1}{2}\dist(\tilde{x}, z)^2.
\]

As in Case 1, using the expression \eqref{eqn:v-exp} of $v[\rho](\tilde{x})$ we then find
\begin{align*}
v[\rho](\tilde{x})\cdot\log_{\tilde{x}}p \geq\int_S g'(\dist(\tilde{x}, z)^2)\dist(\tilde{x}, z)^2 \d\rho(z),
\end{align*}
from which, by a similar argument we conclude that $\rho=\delta_{\tilde{x}}$. This concludes the proof of Part 1.
\medskip

\noindent \underline{{\em Part 2}}: $\calE_{r, p}\supseteq \{\delta_{q}:q\in \overline{B_r(p)}\}$.
We calculate the energy $E[\delta_q]$ from \eqref{eqn:energy} to get
\[
E[\delta_q]=\frac{1}{2}\iint_{S\times S}g(\dist(x, y)^2)\d\delta_q(x)\d\delta_q(y)=\frac{1}{2}g(d(q, q)^2)=\frac{g(0)}{2}.
\]
On the other hand, for any $\rho\in\calP(\overline{B_r(p)})$, we have
\[
E[\rho]=\frac{1}{2}\iint_{\overline{B_r(p)}\times \overline{B_r(p)}}g(\dist(x, y)^2) \d\rho(x) \d\rho(y)\geq\frac{1}{2}\iint_{\overline{B_r(p)}\times \overline{B_r(p)}}g(0)\d\rho(x)\d\rho(y)=\frac{g(0)}{2}.
\]
This implies that $\delta_q$ is a global minimizer of the energy and in particular, a critical point. 
\end{proof}

We combine Lemma \ref{steady} and Proposition \ref{prop:lasalle} to obtain the following theorem.

\begin{theorem}\label{T3.1}
Assume that $M\in\calM_\uc$ for some $\uc\in\bbr$ satisfies (\textbf{M}), and $K$ is a strongly attractive potential. Also assume $\mathrm{supp}(\rho_0)\subset \overline{B_{r}(p)}$ for some $p\in M$ and $0<r<\rw$, and let $\rho_t$ be a weak solution to system \eqref{eqn:model} with initial density $\rho_0$. Then, $\rho_t$ exhibits asymptotic consensus in the following sense:
\begin{align}\label{CIF}
\lim_{t\to\infty}\iint_{\overline{B_r(p)}\times \overline{B_r(p)}}\dist(x, y)\d\rho_t(x)\d\rho_t(y)=0.
\end{align}
\end{theorem}

\begin{proof}
Combine Proposition \ref{prop:lasalle} and Lemma \ref{steady} to obtain
\begin{equation}
\label{eqn:W2-lim}
\lim_{t\to\infty}W_2(\rho_t,\delta_{q(t)})=0,
\end{equation}
for some time-dependent delta measure at $q(t) \in \overline{B_r(p)}$. 

The distance $W_2(\rho_t, \delta_{q(t)})$ is given by
\[
W_2(\rho_t, \delta_{q(t)})=\left(\int_{\overline{B_r(p)}}\dist(x,q(t))^2 \d\rho_t(x)\right)^{1/2}.
\]

Recall by Theorem \ref{thm:wp} that $\mathrm{supp}(\rho_t)\subset\overline{B_r(p)}$ for all $t\geq0$. We apply the Cauchy--Schwartz inequality to get
\begin{align*}
W_2\left(\rho_t, \delta_{q(t)}\right)^2&=\left(\int_{\overline{B_r(p)}} \d\rho_t(x)\right)\left(\int_{\overline{B_r(p)}}\dist(x,q(t))^2 \d\rho_t(x)\right)\\
&\geq\left(
\int_{\overline{B_r(p)}}\dist(x, q(t)) \d\rho_t(x)
\right)^2.
\end{align*}
This yields
\begin{align*}
W_2(\rho_t, \delta_{q(t)})&\geq\int_{\overline{B_r(p)}}\dist(x, q(t)) \d\rho_t(x)\\
&=\frac{1}{2}\left(\int_{\overline{B_r(p)}}\dist(x, q(t))\d\rho_t(x)+\int_{\overline{B_r(p)}}\dist(y, q(t))\d\rho_t(y)\right)\\
&=\frac{1}{2}\iint_{\overline{B_r(p)}\times\overline{B_r(p)}}(\dist(x, q(t))+\dist(y, q(t))\d\rho_t(x) \d\rho_t(y)\\
&\geq\frac{1}{2}\iint_{\overline{B_r(p)}\times\overline{B_r(p)}}\dist(x, y) \d\rho_t(x)\d\rho_t(y),
\end{align*}
where we used triangle inequality at the last step. The conclusion of the theorem follows from the inequality above and \eqref{eqn:W2-lim}.
\end{proof}

\begin{remark}
The consensus result in Theorem \ref{T3.1} is in integral form and does not provide any information on the asymptotic behaviour of the diameter of the support of $\rho_t$ or about the rate of convergence in the limit \eqref{eqn:W2-lim}. A quantitative study is done in Section \ref{sect:conv} under a stricter condition on the potential and on the size of the initial support.
\end{remark}


\section{Convergence rate of the diameter of the support}
\label{sect:conv}

In this section we will assume that the initial density is supported in $\overline{B_r(p)}$ with $p \in M$ and
\[
0<r<\rc,
\]
where $\rc$ is given by
\begin{equation}
\label{eqn:rc}
\rc=\min\left\{
\frac{\inj(M)}{2}, \frac{\pi}{4\sqrt{\mu}}
\right\}.
\end{equation}
We make again the convention that $\frac{1}{\sqrt{\mu}}=\infty$ when $\mu \leq 0$. Note first that $\rc \leq \rw$, hence the well-posedness and asymptotic consensus results (Theorems \ref{thm:wp} and \ref{T3.1}, respectively) continue to hold. Also, if $\mu\leq 0$ and $M$ is simply connected then $\inj(M) = \infty$ by Cartan-Hadamard theorem; consequently, $\rc = \rw=\infty$ in this case.

To obtain the convergence rate of the diameter of $\mathrm{supp}(\rho_t)$, we make the following additional assumption on $K$.
\medskip

\noindent(\textbf{Kc}) $K:M\times M\to \bbr$ is a {\em strongly attractive} potential (Definition \ref{defn:strong-pot}) that satisfies
\[
\begin{cases}
\displaystyle\frac{\theta}{\sin(\sqrt{\uc}\theta)}g'(\theta^2)&\quad\text{is non-decreasing when }\uc>0,\\[10pt]
g'(\theta^2)&\quad\text{is non-decreasing when }\uc\leq0,
\end{cases}
\]
for $0\leq \theta < 2\rc$.

We present first some key lemmas.
\begin{lemma}\label{L4.1}
Let $M\in \calM_{\uc}$ for some $\uc\in\bbr$ and $K$ satisfy (\textbf{M}) and (\textbf{Kc}), respectively. Consider three points $x, y, z\in \overline{B_r(p)}$ for some $p\in M$ and $0<r<\rc$, $x\neq y$. Then, the following inequalities hold for each case. 
\smallskip

\noindent (Case 1: $\uc>0$)
\begin{align}\label{D-4-0}
\begin{aligned}
&\g'(\dist(x, z)^2)\log_xz\cdot\log_xy+\g'(\dist(y, z)^2)\log_yz\cdot\log_yx\\
&\hspace{7cm}\geq \frac{\sin(2\sqrt{\uc}(r_c-r))}{2}\g'(\dist(x, y)^2/4)\dist(x, y)^2.
\end{aligned}
\end{align}
(Case 2: $\uc\leq0$)
\begin{align}\label{D-4-1}
\g'(\dist(x, z)^2)\log_xz\cdot\log_xy+\g'(\dist(y, z)^2)\log_yz\cdot\log_yx\geq \frac{1}{2}g'(\dist(x, y)^2/4) \dist(x, y)^2.
\end{align}

\end{lemma}

\begin{proof}
While this lemma is essential for our main result, its proof is rather lengthy, and we present it in Appendix \ref{app1}.
\end{proof}

The following lemma gives an estimate on how the distance between two characteristic paths that originate from different points evolves in time.
\begin{lemma}\label{flowest}
Assume $M\in \calM_{\uc}$ for some $\uc\in\bbr$ and $K$ satisfy (\textbf{M}) and (\textbf{Kc}), respectively. Consider a weak solution $\rho_t$ of \eqref{eqn:model} with initial data $\mathrm{supp}(\rho_0)\subset \overline{B_r(p)}$ for some $p\in M$ and $0<r<\rc$. Then, for any $x,y \in \textrm{supp }(\rho_0)$, $x \neq y$, we have the following estimate:
\begin{align*}
\frac{\d}{\d t}\dist(\Psi_{v[\rho]}^t(x), \Psi_{v[\rho]}^t(y))
\leq-C_\uc\g'(\dist(\Psi_{v[\rho]}^t(x), \Psi_{v[\rho]}^t(y))^2/4) \, \dist(\Psi_{v[\rho]}^t(x), \Psi_{v[\rho]}^t(y)),
\end{align*}
where $C_\uc$ is given as
\begin{equation}
\label{eqn:Cmu}
C_\uc=\begin{cases}
\sin(2\sqrt{\mu}(\rc-r)),&\qquad\text{when }\mu>0,\\[2pt]
1,&\qquad\text{when }\mu\leq 0.
\end{cases}
\end{equation}
\end{lemma}
\begin{proof}
By chain rule, \eqref{eqn:grad-d2}, and Definition \ref{defn:sol} of a weak solution (see also \eqref{eqn:characteristics-general}, we compute
\begin{align*}
&\frac{\d}{\d t}\dist(\Psi_{v[\rho]}^t(x), \Psi_{v[\rho]}^t(y))^2\\
&=-2\log_{\Psi_{v[\rho]}^t(x)}\Psi_{v[\rho]}^t(y)\cdot \frac{d}{dt} \Psi_{v[\rho]}^t(x)-2\log_{\Psi_{v[\rho]}^t(y)}\Psi_{v[\rho]}^t(x)\cdot\frac{d}{dt}\Psi_{v[\rho]}^t(y)
\\[5pt]
&=-2\log_{\Psi_{v[\rho]}^t(x)}\Psi_{v[\rho]}^t(y)\cdot v[\rho]( \Psi_{v[\rho]}^t(x),t)-2\log_{\Psi_{v[\rho]}^t(y)}\Psi_{v[\rho]}^t(x)\cdot v[\rho](\Psi_{v[\rho]}^t(y),t).
\end{align*}
By using \eqref{eqn:gradK-gen} and the definition of push-forward (see \eqref{eqn:push-forward}), we write the velocity field in \eqref{eqn:v-field} as
\[
v[\rho](\Psi^t_{v[\rho]}(x),t)=2\int_{\mathrm{supp}(\rho_0)}g'(\dist(\Psi^t_{v[\rho]}(x), \Psi^t_{v[\rho]}(z))^2)\log_{\Psi^t_{v[\rho]}(x)}{\Psi^t_{v[\rho]}(z)}\d\rho_0(z).
\]
We combine the two calculations above to get
\begin{align*}
&\frac{\d}{\d t}\dist(\Psi_{v[\rho]}^t(x), \Psi_{v[\rho]}^t(y))^2\\
&=-4\int_{\mathrm{supp}(\rho_0)}\bigg( g'(\dist(\Psi^t_{v[\rho]}(x), \Psi^t_{v[\rho]}(z))^2)\log_{\Psi^t_{v[\rho]}(x)}{\Psi^t_{v[\rho]}(z)}\cdot \log_{\Psi^t_{v[\rho]}(x)}{\Psi^t_{v[\rho]}(y)} \\
&\quad+g'(\dist(\Psi^t_{v[\rho]}(y), \Psi^t_{v[\rho]}(z))^2)\log_{\Psi^t_{v[\rho]}(y)}{\Psi^t_{v[\rho]}(z)}\cdot \log_{\Psi^t_{v[\rho]}(y)}{\Psi^t_{v[\rho]}(x)}\bigg) \d\rho_0(z).
\end{align*}
Now, we apply Lemma \ref{L4.1} for the points $\Psi^t_{v[\rho]}(x)$, $\Psi^t_{v[\rho]}(y)$ and $\Psi^t_{v[\rho]}(z)$. \\

\noindent (Case 1: $\uc>0$)
Use \eqref{D-4-0} to get
\begin{align*}
\frac{\d}{\d t}\dist(\Psi_{v[\rho]}^t(x), \Psi_{v[\rho]}^t(y))^2
\leq-2\sin(2\sqrt{\uc}(\rc-r))\g'(\dist(\Psi_{v[\rho]}^t(x), \Psi_{v[\rho]}^t(y))^2/4)\, \dist(\Psi_{v[\rho]}^t(x), \Psi_{v[\rho]}^t(y))^2,
\end{align*}
and by chain rule,
\begin{align*}
\frac{\d}{\d t}\dist(\Psi_{v[\rho]}^t(x), \Psi_{v[\rho]}^t(y))
\leq-\sin(2\sqrt{\uc}(\rc-r))\g'(\dist(\Psi_{v[\rho]}^t(x), \Psi_{v[\rho]}^t(y))^2/4)\, \dist(\Psi_{v[\rho]}^t(x), \Psi_{v[\rho]}^t(y)).
\end{align*}
\smallskip

\noindent (Case 2: $\uc\leq 0$)
Use \eqref{D-4-1} to get
\begin{align*}
\frac{\d}{\d t}\dist(\Psi_{v[\rho]}^t(x), \Psi_{v[\rho]}^t(y))^2
\leq -2g'(\dist(\Psi_{v[\rho]}^t(x), \Psi_{v[\rho]}^t(y))^2/4)\, \dist(\Psi_{v[\rho]}^t(x), \Psi_{v[\rho]}^t(y))^2.
\end{align*}
Then by chain rule, we find
\begin{align*}
\frac{\d}{\d t}\dist(\Psi_{v[\rho]}^t(x), \Psi_{v[\rho]}^t(y))
\leq -g'(\dist(\Psi_{v[\rho]}^t(x), \Psi_{v[\rho]}^t(y))^2/4)\, \dist(\Psi_{v[\rho]}^t(x), \Psi_{v[\rho]}^t(y)).
\end{align*}
This completes the proof of the lemma.
\end{proof}

Finally, we list an immediate result from ODE theory.
\begin{lemma}\label{L5.2}
Let a time dependent function $\theta$ satisfy the following ODE:
\begin{equation}
\label{eqn:diff-ineq}
\begin{cases}
\displaystyle\frac{\d}{\d t}\theta\leq -C\theta\g'(\theta^2/4),\\[5pt]
\theta(0)=\theta_0>0,
\end{cases}
\end{equation}
where $C$ is a positive constant and $g'(\theta^2/4)>0$ for all $\theta>0$. Then, we have $0<\theta(t)\leq \theta_0$ and 
\begin{align}\label{E-5}
\int_{\theta_0}^{\theta(t)}\frac{\d\xi}{\xi\g'(\xi^2/4)}\leq -Ct, \qquad \text{ for all } t\geq0.
\end{align}
\end{lemma}
\begin{proof}
The proof is immediate. In particular, one can directly obtain \eqref{E-5} by integrating the given differential inequality.
\end{proof}

Define the diameter of $\rho_t$ as
\begin{align}\label{E-5-1}
\Delta(t):=\mathrm{diam}(\mathrm{supp}(\rho_t))=\max_{x, y\in \mathrm{supp}(\rho_t)}\dist(x, y).
\end{align}
The maximum exists since $\dist(\cdot, \cdot)$ is a continuous function on the compact set $\mathrm{supp}(\rho_t)\times\mathrm{supp}(\rho_t)$. The convergence rate of $\Delta(t)$ as $t\to\infty$ is given by the following theorem.

\begin{theorem}\label{thm:rate-conv}
Assume $M\in \calM_{\uc}$ for some $\uc\in\bbr$ and $K$ satisfy (\textbf{M}) and (\textbf{Kc}), respectively. Take a weak solution $\rho_t$ of \eqref{eqn:model} with initial data $\mathrm{supp}(\rho_0)\subset \overline{B_r(p)}$ for some $p\in M$ and $0<r<\rc$. 
Then $\Delta(t)$ defined in \eqref{E-5-1} satisfies the following inequality in integral form:
\begin{align}\label{E-6}
\int_{\Delta(0)}^{\Delta(t)}\frac{\d\xi}{\xi\g'(\xi^2/4)}\leq -C_\uc t, \qquad \textrm{ for all } t \geq 0,
\end{align}
where $C_\uc$ is given by \eqref{eqn:Cmu}.
Furthermore, we have
\begin{equation}
\label{eqn:limz}
\lim_{t\to\infty}\Delta(t)=0.
\end{equation}
\end{theorem}
\begin{proof}
From the definition of the diameter $\Delta$, we have 
\[
\Delta(t)=\max_{x, y\in \mathrm{supp}(\rho_t)}\dist(x, y)=\max_{x, y\in \mathrm{supp}(\rho_0)}\dist(\Psi^t_{v[\rho]}(x), \Psi^t_{v[\rho]}(y)).
\]
Fix an arbitrary time $\T>0$ and two points $\bar{x},\bar{y}\in\mathrm{supp}(\rho_0)$ such that
\[
\Delta(\T)=\dist(\Psi^{\T}_{v[\rho]}(\bar{x}), \Psi^{\T}_{v[\rho]}(\bar{y})).
\]
Set
\[
\theta(t): = \dist(\Psi^t_{v[\rho]}(\bar{x}), \Psi^t_{v[\rho]}(\bar{y})), \qquad \textrm{ for } t \geq 0.
\]
In particular, $\theta(\T) = \Delta(\T)$ and $\theta(0) = \dist(\bar{x},\bar{y})$.

By Lemma \ref{flowest}, $\theta(t)$ satisfies the differential inequality \eqref{eqn:diff-ineq}, with initial value $\theta_0 = d(\bar{x},\bar{y})$.  By Lemma \ref{L5.2}, it then holds that $0<\theta(t) \leq \dist(\bar{x}, \bar{y})$ and
\[
\int_{\dist(\bar{x}, \bar{y})}^{\theta(t)}\frac{\d\xi}{\xi g'(\xi^2/4)}\leq -C_\uc t, \qquad \text{ for all } t \geq 0.
\]
At $t = \T$ one has $\Delta(\T) \leq \dist(\bar{x}, \bar{y})$ and
\begin{equation}
\label{eqn:ineq-tbar}
\int_{\dist(\bar{x}, \bar{y})}^{\Delta(\T)}\frac{\d\xi}{\xi g'(\xi^2/4)}\leq -C_\uc \T.
\end{equation}
Now use  
\[
\Delta(\T) \leq \dist(\bar{x}, \bar{y}) \leq \max_{x, y\in \mathrm{supp}(\rho_0)}\dist(x, y) = \Delta(0),
\]
to write
\begin{equation}
\label{eqn:ineq-Dz}
\int_{\Delta(0)}^{\Delta(\T)}\frac{\d\xi}{\xi\g'(\xi^2/4)}\leq \int_{\dist(\bar{x}, \bar{y})}^{\Delta(\T)}\frac{\d\xi}{\xi\g'(\xi^2/4)},
\end{equation}
where we also used that the integrand is positive. Finally, combine \eqref{eqn:ineq-tbar} and \eqref{eqn:ineq-Dz} to arrive at
\begin{equation*}
\int_{\Delta(0)}^{\Delta(\T)}\frac{\d\xi}{\xi\g'(\xi^2/4)}\leq -C_\uc \T.
\end{equation*}
The inequality above holds for an arbitrary $\T>0$, and this shows \eqref{E-6}.

To show \eqref{eqn:limz} we write
\begin{align}
\int_{\Delta(0)}^{\limsup_{t\to\infty}\Delta(t)}\frac{\d\xi}{\xi g'(\xi^2/4)}&=\int_{\Delta(0)}^{\lim_{t\to\infty}\sup_{\tau\in[t, \infty)}\Delta(\tau)}\frac{\d\xi}{\xi g'(\xi^2/4)} \nonumber \\
&=\lim_{t\to\infty}\int_{\Delta(0)}^{\sup_{\tau\in[t, \infty)}\Delta(\tau)}\frac{\d\xi}{\xi g'(\xi^2/4)} \nonumber \\
&=\lim_{t\to\infty}\left(\sup_{\tau\in[t, \infty)}\int_{\Delta(0)}^{\Delta(\tau)}\frac{\d\xi}{\xi g'(\xi^2/4)}\right), \label{eqn:limsup}
\end{align}
where the last equal sign comes from the fact that $\Delta(t)$ is non-negative and that the function $F(x):=\int_{\Delta(0)}^x \frac{d\xi}{\xi g'(\xi^2/4)}$ is non-decreasing, by which it holds that
\[
F \Bigl( \sup_{\tau\in[t, \infty)} \Delta(\tau) \Bigr) = \sup_{\tau\in[t, \infty)} F(\Delta(\tau)).
\]

By using \eqref{E-6} in \eqref{eqn:limsup}, we then find
\begin{align*}
\int_{\Delta(0)}^{\limsup_{t\to\infty}\Delta(t)}\frac{\d\xi}{\xi g'(\xi^2/4)} 
&\leq\lim_{t\to\infty}\sup_{\tau\in[t, \infty)}(-C_\lc \tau) \\ &= \lim_{t\to\infty}(-C_\lc t) \\ &=-\infty.
\end{align*}

Since the function $\frac{1}{\xi g'(\xi^2/4)}$ has no singularity on $\xi\in (0, \Delta(0)]$, it implies that 
\[ 
\limsup_{t\to\infty}\Delta(t)=0.
\]
Indeed, if $\limsup_{t\to\infty}\Delta(t)>0$, then 
\[
\int_{\Delta(0)}^{\limsup_{t\to\infty}\Delta(t)}\frac{\d\xi}{\xi g'(\xi^2/4)} > -\infty.
\]
Since $\Delta(t)\geq0$, we can now conclude \eqref{eqn:limz}.
\end{proof}

An explicit rate of  convergence can be computed for certain interaction potentials, as given by the following corollary.
\begin{corollary}\label{C4.1}
Under the same assumptions as in Theorem \ref{thm:rate-conv}, assume in addition that
\begin{equation}
\label{eqn:gp}
\g'(\theta^2)\geq \alpha\theta^{\beta-2},\qquad \textrm{ for all } \theta\in(0, 2\rc),
\end{equation}
for some $\alpha>0$ and $\beta\geq2$. Then, we can express the convergence rate of $\Delta(t)$ as follows.
\[
\Delta(t)\lesssim\begin{cases}
\displaystyle O\left(e^{-ct}\right),&\qquad\text{if }\beta=2,\\[5pt]
\displaystyle O\left(t^{-\frac{1}{\beta-2}}\right),&\qquad\text{if }\beta>2,
\end{cases}\qquad\text{as}\quad t\to\infty,
\]
for some $c>0$.
\end{corollary}
\begin{proof}
We substitute the given condition into \eqref{E-6} to get
\[
\int_{\Delta(0)}^{\Delta(t)}\frac{\d\xi}{\alpha \xi(\xi/2)^{\beta-2}}\leq \int_{\Delta(0)}^{\Delta(t)}\frac{\d\xi}{\xi\g'(\xi^2/4)}\leq -C_\uc t.
\]
This yields
\[
\int_{\Delta(0)}^{\Delta(t)}\frac{\d\xi}{\xi^{\beta-1}}\leq -\alpha 2^{2-\beta}C_\uc  t.
\]
If $\beta=2$, then by direct integration we get
\begin{equation}
\label{eqn:beta2}
\Delta(t)\leq \Delta(0) \exp(-\alpha C_\uc t).
\end{equation}
If $\beta>2$, then direct integration yields
\[
\frac{1}{2-\beta}\left(
\frac{1}{\Delta(t)^{\beta-2}}-\frac{1}{\Delta(0)^{\beta-2}}
\right)\leq  -\alpha 2^{2-\beta}C_\uc  t,
\]
which after some further trivial algebra leads to
\begin{equation}
\label{eqn:betag2}
\Delta(t)\leq \Delta(0)\left(1+\alpha(\beta-2) (\Delta(0)/2)^{\beta-2}C_\uc t\right)^{-\frac{1}{\beta-2}}.
\end{equation}
\end{proof}

\begin{example} (Power-law potential)
\label{ex:power-law} 
Consider an attractive interaction potential in power-law form:
\begin{equation}
\label{eqn:p-law}
g(\theta^2)=\frac{1}{\beta}\theta^\beta, \qquad \beta \geq 2.
\end{equation}
Note that $g'(\theta^2)=\frac{1}{2} \theta^{\beta-2}$, so one can use Corollary \ref{C4.1} with $\alpha = \frac{1}{2}$. For $\beta=2$ (quadratic potential) $\Delta(t)$ decays exponentially in time (see \eqref{eqn:beta2}), while for $\beta>2$ it decays algebraically at the rate $\frac{1}{t^{\beta-2}}$ (see \eqref{eqn:betag2}). These decay rates are demonstrated numerically in Section \ref{sect:numerics} for $M=SO(3)$, the $3$-dimensional special orthogonal group.

Potentials in power-law form have been considered in many works on the interaction model \eqref{eqn:model} in Euclidean spaces \cite{BaCaLaRa2013, FeHu13, FeHuKo11, KoSuUmBe2011, Brecht_etal2011}, as well as for the model set up on Riemannian manifolds \cite{FeHaPa2021, FeZh2019}.  In particular, attractive-repulsive interaction potentials were shown to lead to complex equilibrium configurations, supported on sets of various dimensions \cite{Balague_etalARMA, KoSuUmBe2011, Brecht_etal2011}. Also, the existence and characterization of minimizers of the interaction energy with interactions in power-law form have been an active research topic in recent years \cite{ChFeTo2015, SiSlTo2015, CaCaPa2015}.
\end{example}

\begin{remark}
Consider a potential $K$ made of two parts: $K(x, y)=K_1(x, y)+K_2(x, y)$, with $K_1(x, y)=\g_1(\dist(x, y)^2)$ and $K_2(x, y)=\g_2(\dist(x, y)^2)$, such that $g_1$ and $g_2$ satisfy
\[
\g_1'(\theta^2)\leq \gamma \g_2'(\theta^2), \qquad \textrm{ for all } 0\leq \theta\leq\rc,
\]
for some constant $\gamma>0$. The relationship above states that the attraction modelled by $K_2$ is stronger than that of $K_1$. Then we have
\[
 \int_{\Delta(0)}^{\Delta(t)}\frac{\d\xi}{\xi \g'(\xi^2/4)}=\int_{\Delta(0)}^{\Delta(t)}\frac{\d \xi}{\xi \g_1'(\xi^2/4)+\xi \g_2'(\xi^2/4)}\leq\frac{1}{1+\gamma}\int_{\Delta(0)}^{\Delta(t)}\frac{\d\xi}{\xi g_2'(\xi^2/4)},
\]
where $g = g_1+g_2$.

Consequently, the decay rate of $\Delta(t)$ is determined by the rate of the stronger potential $K_2$. For example, if $\g(\theta^2)=\frac{1}{2}\theta^2+\frac{1}{4}\theta^4$, then we can set $g_1(\theta^2) = \frac{1}{4}\theta^4$, $g_2(\theta^2)=\frac{1}{2}\theta^2$, and conclude that $\Delta(t)$ converges to zero exponentially fast. 
\end{remark}

\begin{remark} 
\label{remark:compare}
We compare here Theorem \ref{thm:rate-conv} above with the asymptotic consensus in \cite[Theorem 5.18]{FeHaPa2021}, the most general result available prior to the present work.

\noindent(i) {\em Classes of manifolds:} In \cite[Theorem 5.18]{FeHaPa2021} the authors only considered manifolds $M$ of constant sectional curvature, while Theorem \ref{thm:rate-conv} applies to general manifolds of bounded curvature.  
\smallskip

\noindent(ii) {\em Assumptions on the potential:} By \cite[Proposition 5.17]{FeHaPa2021}, an initial density supported in $\overline{B_r(p)}$ achieves asymptotic consensus provided $g$ satisfies the following assumptions: $g' \geq 0$, $g'(0) \geq \alpha$ for a positive constant $\alpha$, $g'$ is $C^1$, and
\[
\begin{cases}
\displaystyle\frac{\theta}{\sin(\theta)}g'(\theta^2)&\text{ is non-decreasing when } \secc = 1, \\[10pt]
g'(\theta^2)&\text{ is non-decreasing when } \secc =0,\\[3pt]
\displaystyle\frac{\theta}{\sinh(\theta)}g'(\theta^2)&\text{ is non-decreasing when } \secc = -1,
\end{cases}
\]
for $0<\theta<2r$, with $r< \mathrm{conv }(M)$ when $\secc = -1$ or $0$ and $r< \mathrm{min} \{\mathrm{conv}(M),\frac{\pi}{4}\}$ when $\secc = 1$ (here,  
$\mathrm{conv }(M)$ denotes the convexity radius of $M$).

In contrast, in Theorem \ref{thm:rate-conv} we have only assumed that $g$ satisfies $g'>0$, and 
\[
\begin{cases}
\displaystyle\frac{\theta}{\sin(\sqrt{\mu}\theta)}g'(\theta^2)&\text{ is non-decreasing when }\mu>0,\\[10pt]
g'(\theta^2)&\text{ is non-decreasing when }\mu\leq 0,
\end{cases}
\]
for $0<\theta<2 \rc$. In particular, compared to \cite[Proposition 5.17]{FeHaPa2021} we dropped the very restrictive condition $g'(0)\geq\alpha>0$, which rules out for instance the important class of power-law potentials discussed in Example \ref{ex:power-law}. Also,  for manifolds of negative curvature, the assumption in Theorem \ref{thm:rate-conv} on the monotonicity of $g'$ is weaker than the corresponding assumption in \cite[Proposition 5.17]{FeHaPa2021} (as $\sinh(\theta)/\theta$ is non-decreasing).
\smallskip

\noindent(iii) {\em Rate of convergence:} There is no argument in \cite[Theorem 5.18]{FeHaPa2021} on the convergence rate of the diameter $\Delta(t)$. 
\end{remark}


\section{Weakly attractive potentials: asymptotic behaviour}
\label{sect:weak-att}
In this section we investigate the asymptotic behaviour of solutions to \eqref{eqn:model} when the strong attraction condition \eqref{eqn:Katt-s} is replaced by a weaker one. Consider the following definition.
\begin{definition}[Weakly attractive potential]\label{defn:weak-pot}
An interaction potential $K$ is called {\em weakly attractive} if it satisfies assumption (\textbf{K}), with \eqref{eqn:Katt} in the form
\begin{align}\label{eqn:Katt-w}
\begin{cases}
\g'(r^2)=0,&\qquad \text{ for all } 0\leq r\leq \re,\\[2pt]
\g'(r^2)>0,&\qquad \text{ for all } \re<r<\infty,
\end{cases}
\end{align}
for some $0<\re < \frac{r_w}{2}$.
\end{definition}
Equation \eqref{eqn:Katt-w} implies that two points within distance $\re$ do not interact with each other,  while points that are further than $\re$ apart feel a non-trivial attractive interaction. 

To apply Proposition \ref{prop:lasalle}, we express $\calE_{r, p}$ introduced in \eqref{SE}. The lemma below is the analogue of Lemma \ref{steady} for weakly attractive potentials. 
\begin{lemma}\label{L5.1}
Assume that $M\in\calM_\uc$ with $\uc\in\bbr$ satisfies (\textbf{M}) and $K$ is a weakly attractive potential. Then, for all $0<r<\rw$ and $p\in M$, we have
\[
\calE_{r,p}=\{\rho\in \calP(\overline{B_r(p)}):\mathrm{diam}(\mathrm{supp}(\rho))\leq \re \}.
\]
\end{lemma}
\begin{proof}
The proof is very similar to that of Lemma \ref{steady}, we only sketch it here.

\noindent \underline{\em Part 1}: $\calE_{r,p}\subseteq\{\rho\in \calP(\overline{B_r(p)}):\mathrm{diam}(\mathrm{supp}(\rho))\leq \re\}$. Let $\rho$ be a steady state of system \eqref{eqn:model}, and consider the notations from the proof of Lemma \ref{steady}. Specifically, $R:=\max_{x\in \supp}\dist(x, p)$, with $0\leq R\leq r<\rw$, and $\tilde{x}\in \supp$ satisfies $\dist(\tilde{x}, p)=R$. Following the argument made in the proof of Lemma \ref{steady}, one can then show \eqref{eqn:int-zero}.
Hence, we have $g'(\dist(\tilde{x}, z)^2)=0$ or $z=\tilde{x}$ for $z \in \supp$, a.e. with respect to $\rho$. This implies that
\[
\dist(\tilde{x}, z)\leq \re, \qquad \text{ for }z \in \supp, \text{ a.e. with respect to }\rho.
\]

On the other hand, for any $x, y\in \supp$, we have
\[
\dist(x, y)\leq \dist(x, \tilde{x})+\dist(y, \tilde{x})\leq 2\re < r_w.
\]
Take two points $y_1$ and $y_2$ in $\supp$ such that
\[
\dist(y_1, y_2)=D:=\max_{x, y\in \supp}\dist(x, y) < r_w.
\]
Therefore, $ \supp \subset \overline{B_D(y_1)}$, with $D<\rw$. Then, from a similar argument that we used to prove \eqref{eqn:int-zero}, we get
\[
\int_{\supp}  g'(\dist(y_1, z)^2)\dist(y_1, z)^2 \d\rho(z)=0,
\] 
which implies 
\[
\dist(y_1, z)\leq \re, \qquad \text{ for } z \in \supp, \text{ a.e. with respect to }\rho. 
\]
From here we conclude
\[
\mathrm{diam}(\supp) = \dist(y_1, y_2)  \leq \re.
\]
\medskip

\noindent \underline{\em Part 2}: $\calE_{r,p}\supseteq\{\rho\in \calP(\overline{B_r(p)}):\mathrm{diam}(\mathrm{supp}(\rho))\leq \re\}$. Also by a similar argument used in the proof of Lemma \ref{steady}, it can be easily inferred that provided $\rho$ satisfies
\[
\mathrm{diam}(\mathrm{supp}(\rho))\leq \re,
\]
then $\rho$ is a global minimizer of the energy, and hence a critical point.
\end{proof}

The following lemma is used in proving the consensus result.
\begin{lemma}[\cite{ambrosio2013user}, Corollary 2.22]\label{wageod}
Let $(\rho_\tau)$ be a curve in $\mathcal{P}_2(M)$, $0\leq \tau \leq 1$. Then the following two statements are equivalent:\\[2pt]
\noindent(i) $(\rho_\tau)$ is a geodesic in $(\mathcal{P}_2(M), W_2)$ that joins $\rho_0$ and $\rho_1$, \\[3pt]
\noindent(ii) there exists a plan $\gamma\in \mathcal{P}(TM)$ ($TM$ denotes the tangent bundle of $M$) such that
\[
\int_{TM} \|v\|^2 \d\gamma(x, v)=W_2^2(\rho_0, \rho_1),\qquad (\mathrm{Exp}(\tau))_\#\gamma=\rho_\tau, \quad \text{ for all } 0 \leq \tau \leq 1,
\]
where $\mathrm{Exp}(\tau): TM \to M$ is the map $(x, v)\mapsto \exp_x(\tau v)$.
\end{lemma}

The main result of this section is the following theorem.
\begin{theorem}\label{Convt}
Assume that $M\in\calM_\uc$ with $\uc\in\bbr$ satisfies (\textbf{M}) and $K$ is a weakly attractive potential. Also assume $\mathrm{supp}(\rho_0)\subset \overline{B_{r}(p)}$ for some $p\in M$ and $0<r<\rw$, and let $\rho_t$ be a weak solution to system \eqref{eqn:model} with initial density $\rho_0$. Then, $\rho_t$ exhibits the following asymptotic behaviour:
\[
\lim_{t\to\infty}\iint_{\dist(x, y)> \re}\left(\g(\dist(x, y)^2)-g(0)\right) \d\rho_t(x) \d\rho_t(y)=0.
\]
\end{theorem}
\begin{proof}
We combine Proposition \ref{prop:lasalle} and Lemma \ref{L5.1} to obtain
\begin{equation}
\label{eqn:limW2}
\lim_{t\to\infty}W_2(\rho_t, \tilde{\rho}_t)=0,
\end{equation}
for some time dependent measure $\tilde{\rho}_t\in\calP(\overline{B_r(p)})$ which satisfies $\mathrm{diam}(\mathrm{supp}(\tilde{\rho}_t))\leq \re$. 

By Lemma \ref{wageod}, for $t>0$ fixed, there exists a plan $\gamma\in \mathcal{P}(TM)$ such that
\[
\int_{TM}\|v\|^2 \d\gamma(x, v)=W_2^2(\rho_t, \tilde{\rho}_t),\quad (\mathrm{Exp}(0))_\#\gamma=\rho_t,\quad (\mathrm{Exp}(1))_\#\gamma=\tilde{\rho}_t.
\] 

Define the following functional of $\rho\in\mathcal{P}(M)$:
\[
F[\rho]:=\iint_{M\times M}(\g(\dist(x, y)^2)-\g(0)) \d\rho(x) \d\rho(y).
\]
Using the plan $\gamma$ we express $F[\rho_t]$ and $F[\tilde{\rho}_t]$ as 
\[
F[\rho_t]=\iint_{TM\times TM}(\g(\dist(x, y)^2)-\g(0)) \d\gamma(x, v) \d\gamma(y, u),
\]
and
\[
F[\tilde{\rho}_t]=\iint_{TM\times TM}(\g(\dist(\exp_x(v), \exp_y(u))^2)-\g(0)) \d\gamma(x, v) \d\gamma(y, u).
\]

Since $\mathrm{diam}(\mathrm{supp}(\tilde{\rho}_t))\leq \re$, we infer that $\dist(\exp_x(v), \exp_y(u))\leq \re$ a.e. with respect to $\gamma\otimes \gamma$. Also, from the assumption \eqref{eqn:Katt-w} on $\g$, we have
\begin{equation}
\label{eqn:gdiff}
\g(\dist(\exp_x(v), \exp_y(u))^2)-\g(0)=0,
\end{equation}
a.e. with respect to $\gamma\otimes \gamma$. Consequently, $F[\tilde{\rho}_t]=0$.  We use this fact to compute
\begin{align}
\label{eqn:Frhot}
F[\rho_t]&=F[\rho_t]-F[\tilde{\rho}_t] \nonumber\\
&=\iint_{TM\times TM}(\g(\dist(x, y)^2)-\g(0)) \d\gamma(x, v) \d\gamma(y, u) \nonumber \\
&\quad-\iint_{TM\times TM}\left(\g(\dist(\exp_x(v), \exp_y(u))^2)-\g(0) \right) \d\gamma(x, v) \d\gamma(y, u) \nonumber \\
&=\iint_{TM\times TM} \left(\g(\dist(x, y)^2)-\g(\dist(\exp_x(v), \exp_y(u))^2) \right) \d\gamma(x, v) \d\gamma(y, u) \nonumber \\
&=\iint_{TM\times TM}\left|\g(\dist(x, y)^2)-\g(\dist(\exp_x(v), \exp_y(u))^2)\right| \d\gamma(x, v) \d\gamma(y, u),
\end{align}
where in the last equality we used \eqref{eqn:gdiff} to get
\[
\g(\dist(x, y)^2)-\g(\dist(\exp_x(v), \exp_y(u))^2)=\g(\dist(x, y)^2)-\g(0)\geq0,
\]
a.e. with respect to $\gamma\otimes \gamma$. 

From the mean value theorem, we have
\begin{equation}
\label{eqn:MVT}
\left|\g(\dist(x, y)^2)-\g(\dist(\exp_x(v), \exp_y(u))^2)\right|= \g'(\eta(x, y, v, u)) \left| \dist(x, y)^2-\dist(\exp_x(v), \exp_y(u))^2 \right|,
\end{equation}
where $\eta$ lies between $\dist(x, y)^2$ and $\dist(\exp_x(v), \exp_y(u))^2$. Since both squared distances are less than $(2\rw)^2$, we have
\begin{equation}
\label{eqn:gp-bound}
g'(\eta(x, y, v, u))\leq \sup_{0\leq \theta\leq 2\rw}g'(\theta^2)=:\mathcal{C}.
\end{equation}
Then, we can combine \eqref{eqn:Frhot}, \eqref{eqn:MVT} and \eqref{eqn:gp-bound} to estimate $F[\rho_t]$ as
\begin{equation}
\label{eqn:Frhot-ineq}
F[\rho_t]\leq \mathcal{C}\iint_{TM\times TM}\left|\dist(x, y)^2-\dist(\exp_x(v), \exp_y(u))^2\right| \d\gamma(x, v) \d\gamma(y, u).
\end{equation}

Now write
\[
\left|\dist(x, y)^2-\dist(\exp_x(v), \exp_y(u))^2\right|=\left|\dist(x, y)-\dist(\exp_x(v), \exp_y(u))\right|\cdot\left|\dist(x, y)+\dist(\exp_x(v), \exp_y(u))\right|.
\]
Since
\[
\dist(x, y)\leq 2\rw\quad\text{and}\quad \dist(\exp_x(v), \exp_y(u))\leq 2\rw,
\]
for a.e. $(x, v, y, u)\in TM^2$ with respect to $\gamma\otimes \gamma$, we have
\[
\left|\dist(x, y)+\dist(\exp_x(v), \exp_y(u))\right|\leq 4\rw,
\]
for a.e. $(x, v, y, u)\in TM^2$ with respect to $\gamma\otimes \gamma$. 

By triangle inequality, we find
\begin{align*}
\left|\dist(x, y)-\dist(\exp_x(v), \exp_y(u))\right|&\leq \left|\dist(x, y)-\dist(\exp_x(v), y)\right|+\left|\dist(\exp_x(v), y)-\dist(\exp_x(v), \exp_y(u))\right|\\[2pt]
&\leq \dist(x, \exp_x(v))+\dist(y, \exp_y(u)) \\[2pt]
& =\|v\|+\|u\|.
\end{align*}
Combining these estimates in \eqref{eqn:Frhot-ineq} we then get
\begin{align*}
F[\rho_t]&\leq 4\rw\mathcal{C}\iint_{TM\times TM}(\|v\|+\|u\|)\d\gamma(x, v) \d\gamma(y, u)\\
&=4\rw\mathcal{C}\iint_{TM\times TM}\|v\|\d\gamma(x, v) \d\gamma(y, u)+4\rw\mathcal{C}\iint_{TM\times TM}\|u\|\d\gamma(x, v) \d\gamma(y, u)\\
&=4\rw\mathcal{C}\int_{TM}\|v\| \d\gamma(x, v) +4\rw\mathcal{C}\int_{TM}\|u\| \d\gamma(y, u)\\
&=8\rw\mathcal{C}\int_{TM}\|v\| \d\gamma(x, v)\\
&\leq8\rw\mathcal{C}\left(\int_{TM} \d\gamma(x, v) \right)^{1/2} \left( \int_{TM}\|v\|^2 \d\gamma(x, v)\right)^{1/2}\\[2pt]
&= 8\rw\mathcal{C}W_2(\rho_t, \tilde{\rho}_t),
\end{align*}
where for the last inequality sign we used Cauchy--Schwartz inequality.

The estimate above holds for any time (note that the constant $\mathcal{C}$ does not depend on $t$). Hence, also using \eqref{eqn:limW2}, we find
\[
0\leq\lim_{t\to\infty}F[\rho_t]\leq \lim_{t\to\infty}  8\rw\mathcal{C}W_2(\rho_t, \tilde{\rho}_t)=0.
\]
We infer that $\lim_{t\to\infty}F[\rho_t]=0$, which is equivalent to
\[
\lim_{t\to\infty}\iint_{M\times M}(\g(\dist(x, y)^2)-\g(0))\d\rho_t(x) \d\rho_t(y)=0.
\]
Finally, since $\g(\dist(x, y)^2)-\g(0)=0$ when $\dist(x, y)\leq \re$, we reach the claimed result:
\begin{align*}
\lim_{t\to\infty}\iint_{\dist(x, y)> \re}(\g(\dist(x, y)^2)-\g(0)) \d\rho_t(x) \d\rho_t(y) &=\lim_{t\to\infty}\iint_{M\times M}(\g(\dist(x, y)^2)-\g(0)) \d\rho_t(x) \d\rho_t(y) \\
&=0.
\end{align*}
\end{proof}

\begin{remark}
Given that $g(\theta^2)>g(0)$ for $\theta > \re$, the asymptotic result in Theorem \ref{Convt} states that $\dist(x,y)\leq \re$ a.e. with respect to $\rho_t \otimes \rho_t$, as $t \to \infty$.
\end{remark}


\section{Numerical results}
\label{sect:numerics}
For numerical simulations we will use the discrete version of model \eqref{eqn:model}. For this purpose, take a positive integer $\N$ and consider a collection of masses $(m_i)_{i=1}^{\N} \subset (0,1)$ such that $\sum_{i=1}^{\N} m_i = 1$, and points $(x_{i}^0)_{i=1}^{\N} \subset M$. Also take an initial density $\rho^{\N}_0$ consisting of $\N$ delta masses supported at these points, i.e., 
\begin{equation}
\label{eq:atomic-initial}
    \rho^{\N}_0 =  \sum_{i=1}^{\N} m_i \delta_{x_{i}^0}.
\end{equation}

The unique weak solution $\rho^{\N} :[0,T) \to \calP(M)$ of \eqref{eqn:model} (in the sense of Definition \ref{defn:sol}), with initial density $\rho^{\N}_0$, is the empirical measure associated to masses $m_i$ and trajectories $(x_i(t))_{i=1}^{\N} \subset M$, i.e., 
\begin{equation}
\label{eq:atomic}
	\rho_t^{\N} = \sum_{i=1}^{\N} m_i \delta_{x_i(t)}, \qquad \mbox{for all $t\in [0,T)$},
\end{equation}
where the trajectories $x_{i}(t)$, $i=1,\dots,\N$, satisfy
\begin{equation}
\label{eq:characteristics-particles}
    \begin{cases}
       x_i'(t) = v[\rho^{\N}](x_i(t)),\\[2pt]
      x_i(0) = x_{i}^0.
   \end{cases}
\end{equation}
We also note here that in the discrete case, the convolution in the expression for the velocity $v[\rho^\N]$ (see \eqref{eqn:model} and \eqref{B-2}) reduces to the finite sum:
\begin{equation}
\label{eqn:v-discrete}
 v[\rho^{\N}](x_i(t)) = - \sum_{j=1}^{\N} m_i\nabla_{x_i} K(x_i(t),x_j(t)).
\end{equation}

The numerical simulations we present below are for $M=SO(3)$, the $3$-dimensional special orthogonal group (also referred to here as the rotation group), given by
\[
SO(3) = \{ R \in \bbr^{3 \times 3}: R^T R = I \text{ and } \text{det }R = 1 \}.
\]
The rotation group is the configuration space of a rigid body in $\bbr^3$ that undergoes rotations only (no translations) and has many applications in engineering, in particular in robotics \cite{TronAfsariVidal2012}. It is topologically nontrivial, as it is not simply connected, it has constant sectional curvature $\secc=1/4$ and radius of injectivity equal to $\pi$.

We parametrize $SO(3)$ using the angle-axis representation. In this parametrization, a rotation matrix $R \in SO(3)$ is identified via the exponential map with a pair $(\theta,\bv) \in [0,\pi] \times S^2$, where $S^2$ denotes the unit sphere in $\bbr^3$. The unit vector $\bv$ indicates the axis of rotation and $\theta$ represents the angle of rotation (by the right-hand rule) about the axis. Using the angle-axis representation we solve numerically the ODE system \eqref{eq:characteristics-particles} (with velocities given by \eqref{eqn:v-discrete}) for the evolution of $\N$ rotation matrices $R_i(t)$ (see \cite{FeHaPa2021} for details on the numerical implementation). We take all particles to have identical masses, i.e., $m_i = 1/\N$, $i=1,\dots,\N$. For time integration we use the 4th order Runge-Kutta method.

For plotting purposes we identify $SO(3)$ with a ball in $\bbr^3$ of radius $\pi$ centred at the origin. The center of the ball corresponds to the identity matrix $I$. A generic point within this ball represents a rotation matrix, with rotation angle given by the distance from the point to the centre, and axis given by the ray from the centre to the point. By this representation, antipodal points on the surface of the ball are identified, as they represent the same rotation matrix (rotation by $\pi$ about a ray gives the same result as rotation by $\pi$ about the opposite ray).

Figure \ref{fig:consensus} illustrates the formation of asymptotic consensus on $SO(3)$ for an interaction potential in power-law form  \eqref{eqn:p-law}. As shown in Example \ref{ex:power-law}, the decay rates of the diameter can be computed explicitly for such potential. The numerical results correspond to a simulation using $\N=40$ particles (i.e., rotation matrices) initialized as follows.  The rotation angles $\theta_i$ were selected randomly in the interval $(0,\pi/4)$, while the unit vectors $\bv_i$ were generated in spherical coordinates, with the polar and azimuthal angles drawn randomly in the intervals $(0,\pi)$ and $(0,2 \pi)$, respectively. By this initialization, all rotation matrices $R_i$ at time $t=0$ are within distance $\pi/4$ from the identity matrix and hence, the assumptions of Theorem \ref{thm:rate-conv} are satisfied.

The plots in Figure \ref{fig:consensus}(a) and (b) correspond to the quadratic potential ($\beta=2$). In Figure \ref{fig:consensus}(a) the initial particles are indicated by black dots and the consensus point by a red diamond. For visualization purposes we do not show the full ball of radius $\pi$ there, but set the axis limits to $[-1,1]$.  Figure \ref{fig:consensus}(b) shows a semi-log plot of the diameter $\Delta$ of the configuration $\{R_i\}_{i=1}^N$ over time, demonstrating the exponential decay \eqref{eqn:beta2}. The rate of decay (slope of the line) is approximately $-1.009$. Figure \ref{fig:consensus}(c) illustrates the decay rate \eqref{eqn:betag2} of the diameter $\Delta$ for various values of the exponent $\beta$ ($\beta = 3,4$, and $8$). The figure shows a log-log plot of the diameter over time, where a linear fit on the last quarter of the numerical run shows decay rates (slopes) that match to two decimal places the analytical rates of $-1$, $-1/2$ and $-1/6 \approx -0.1666$  from  \eqref{eqn:betag2}.

\begin{figure}[!htbp]
 \begin{center}
 \begin{tabular}{ccc}
 \includegraphics[width=0.35\textwidth]{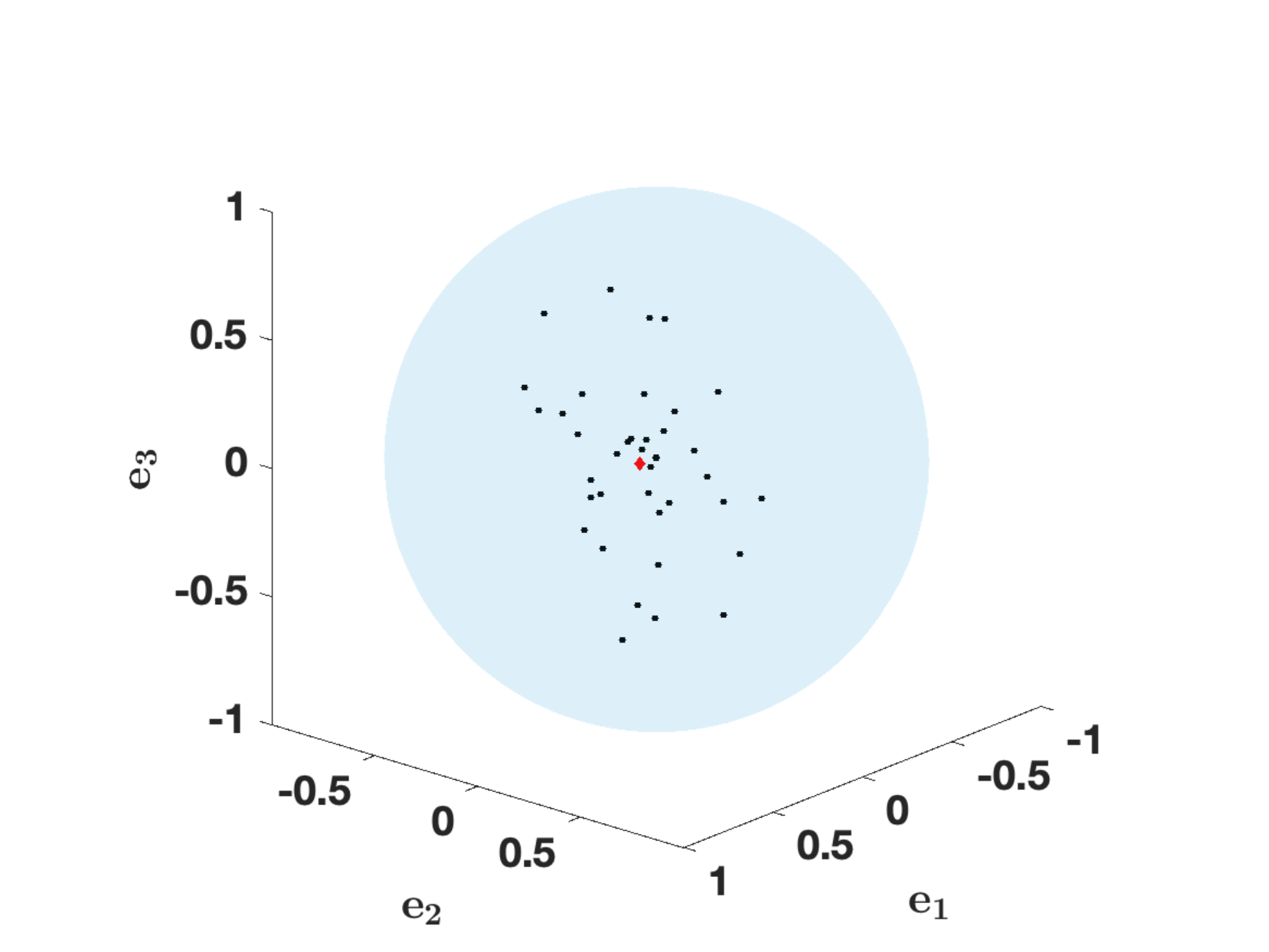} & 
 \includegraphics[width=0.3\textwidth]{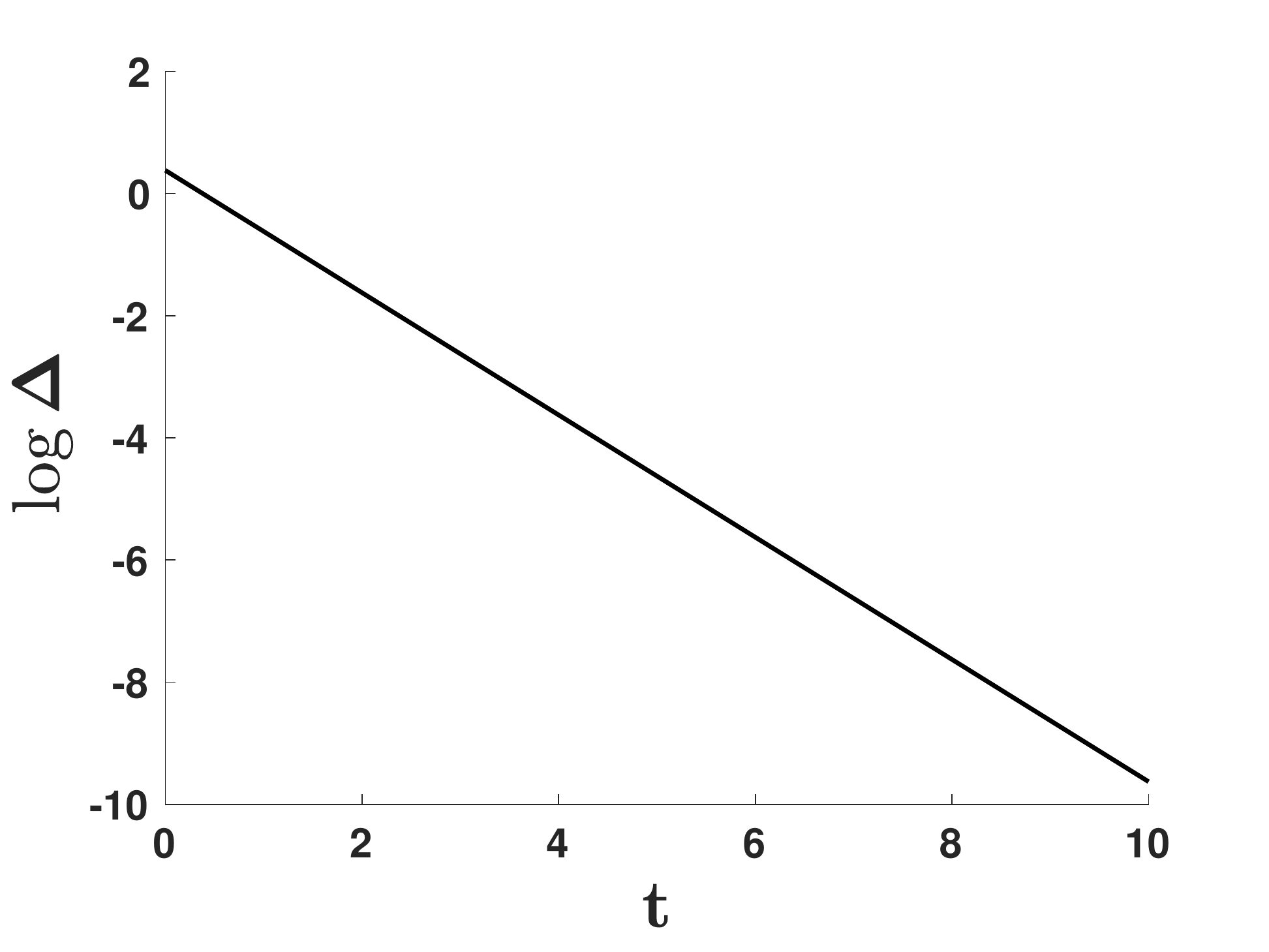}  &
 \includegraphics[width=0.3\textwidth]{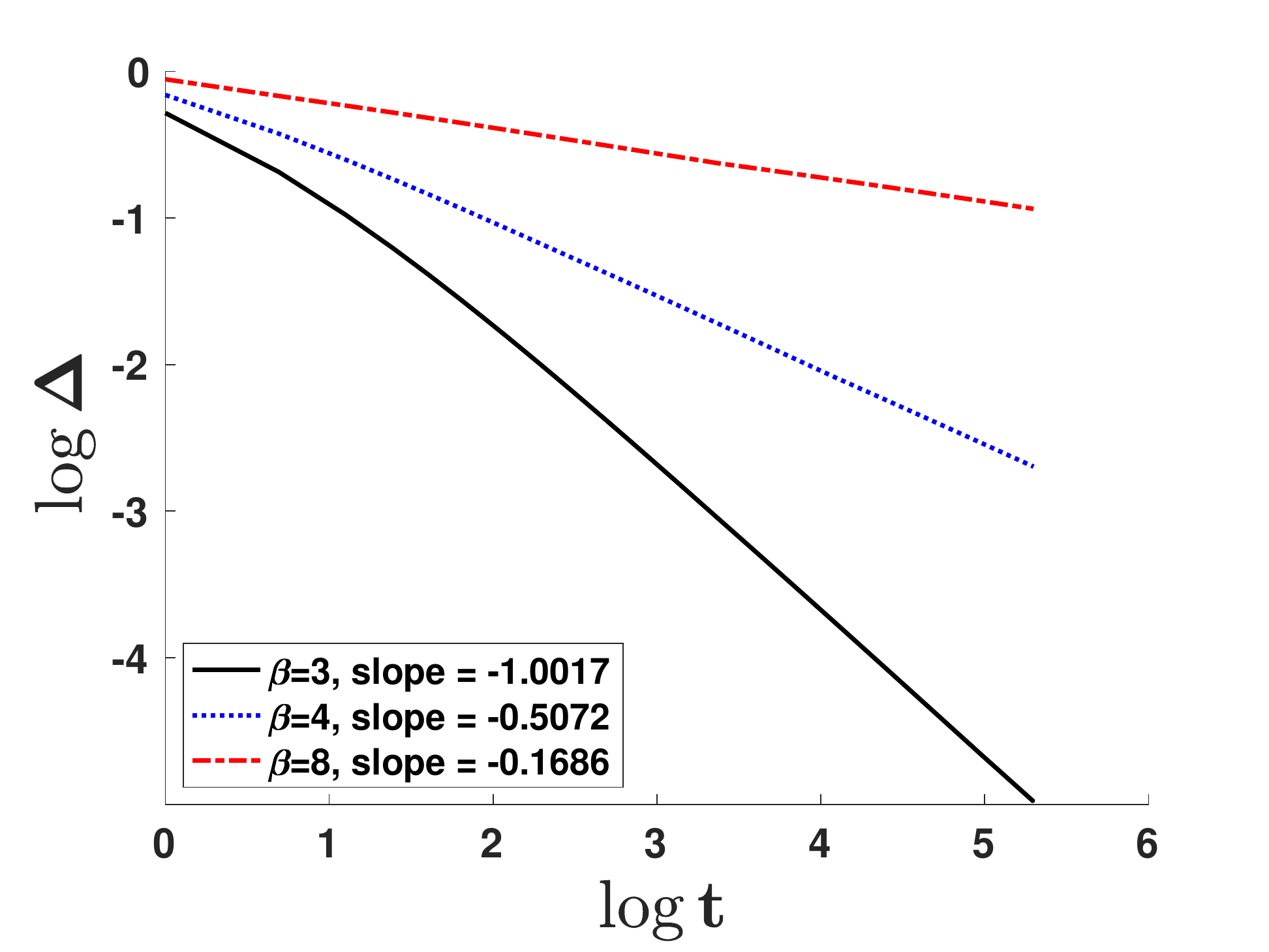} \\
 (a) & (b) & (c)  \end{tabular}
 \begin{center}
 \end{center}
\caption{Asymptotic consensus on the rotation group $SO(3)$, for $\N=40$ matrices with the attractive power-law potential \eqref{eqn:p-law}. (a) Quadratic potential ($\beta=2$). Initial particles are indicated by black dots, and the consensus location is shown with red diamond. (b) Same simulation as in (a), showing the exponential decay of the diameter in a semi-log plot -- see Example \ref{ex:power-law} and equation \eqref{eqn:beta2}. (c) Power-law decay of the diameter in a log-log plot, for different values of $\beta$ - see Example \ref{ex:power-law} and equation \eqref{eqn:betag2}. The numerical rates of decay, as indicated in the legend, match exactly, in their first two digits, the analytical rates $-1$, $-1/2$ and $-1/6$ from \eqref{eqn:betag2}.}
\label{fig:consensus}
\end{center}
\end{figure}


\appendix

\section{Some comments on Theorem \ref{thm:wp}}
\label{appendix:wp}

As noted after Theorem \ref{thm:wp},  the well-posedness result in \cite[Theorem 5.1]{FePa2021} is more restrictive with respect to the size of the set $U$. Specifically, it is assumed there that 
\begin{equation}
\label{eqn:diam-rest1}
\Delta=\mathrm{diam}(U)<\frac{\pi}{2\sqrt{\uc}}.
\end{equation}
The restriction \eqref{eqn:diam-rest1} comes from how the bound on the Hessian of the squared distance function is derived. The authors of \cite{FePa2021} use \cite[Theorem 6.6.1]{Jost2017}, which states that provided 
\begin{equation}
\label{eqn:K-bounds}
\lambda \leq \secc \leq \mu, \qquad \text{ on $B_r(z)$},
\end{equation}
with $\lambda \leq 0 \leq \mu$, then
\begin{equation}
\label{eqn:bHessian}
2\sqrt{\mu}\, d(x,z) \cot( \sqrt{\mu}\, d(x,z)) \|v\| ^2 \leq \langle \mathrm{Hess}_v d_z^2(x),v \rangle \leq 2 \sqrt{-\lambda}\, d(x,z) \coth( \sqrt{-\lambda}\, d(x,z)) \|v\|^2,
\end{equation}
for all $x \in B_r(z)$ and $v \in T_x M$. Here, $d_z$ denotes the distance function to point $z$. Then, the authors of \cite{FePa2021} restrict the diameter of $U$ as in \eqref{eqn:diam-rest1}
to make the left-hand-side of \eqref{eqn:bHessian} nonnegative and bound $\mathrm{Hess }\, \dist_z^2$ by
\[
|\mathrm{Hess}\dist_z^2|\leq L:=2\sqrt{-\lc}\Delta \coth(-\sqrt{\lc}\Delta).
\]

However, by assuming the weaker condition on $\Delta$:
\[
\Delta=\mathrm{diam}(U)<\frac{\pi}{\sqrt{\uc}},
\]
one can bound $\mathrm{Hess} \, \dist_z^2$ by
\[
|\mathrm{Hess} \, \dist_z^2|\leq L':=\max\left\{2\sqrt{-\lc}\Delta \coth(-\sqrt{\lc}\Delta),
\left|2\sqrt{\uc}\Delta\cot(\sqrt{\uc}\Delta)
\right|
\right\}.
\]
All the arguments in the proof of \cite[Theorem 5.1]{FePa2021} would then follow with this more relaxed bound on $\mathrm{diam}(U)$. Theorem \ref{thm:wp} reflects this extension.


\section{Proof of Lemma \ref{L2.1}}\label{appendix:L2}
Let $(\rho_n)_{n=1}^\infty\subset \calP_2(\overline{B_r(p)})$ converge to $\rho\in \calP_2(M)$ with respect to the metric $W_2$, i.e.,
\[
\lim_{n\to\infty}W_2(\rho_n, \rho)=0.
\]
We have, for any $\epsilon >0$,
\begin{align*}
W_2(\rho_n, \rho)^2&=\inf_{\gamma\in\Pi(\rho_n, \rho)}\iint_{M\times M}\dist(x, y)^2 \d\gamma(x, y)\\
&\geq \inf_{\gamma\in \Pi(\rho_n, \rho)} \iint_{M\times (B_{r+\epsilon}(p))^c}\dist(x, y)^2 \d\gamma(x, y)\\
&\geq  \inf_{\gamma\in \Pi(\rho_n, \rho)}  \iint_{M\times (B_{r+\epsilon}(p))^c}\epsilon^2 \d\gamma(x, y)\\[2pt]
&=\epsilon^2 \rho( B_{r+\epsilon}(p)^c)\\
&=\epsilon^2\left(1-\rho(B_{r+\epsilon}(p))\right).
\end{align*}

Since $\epsilon>0$ is independent of $n$, we have
\[
0=\lim_{n\to\infty}W_2(\rho_n, \rho)^2\geq \epsilon^2\left(1-\rho(B_{r+\epsilon}(p))\right)\geq0,
\]
and this yields
\[
1-\rho(B_{r+\epsilon}(p))=0\quad \Longleftrightarrow\quad \rho(B_{r+\epsilon}(p))=1,
\]
for all $\epsilon>0$. 

We then get
\[
\rho(\overline{B_{r}(p)})=\rho\left(\bigcap_{n=1}^\infty B_{r+1/n}(p)\right)=\lim_{n\to\infty}\rho(B_{r+1/n}(p))=1.
\]
This yields $\rho\in \calP_2(\overline{B_r(p)})$ which implies that $\calP_2(\overline{B_r(p)})$ is compact.


\section{Proof of Lemma \ref{L4.1}}\label{app1}
Since $x, y, z\in \overline{B_r(p)}$ and $0<r<\rc=\min\left\{\frac{\inj(M)}{2},\frac{\pi}{4\sqrt{\mu}}\right\}$, we have
\begin{align}\label{D-4-3}
\dist(x, y), \dist(y, z), \dist(z, x)\leq2r=2(r_c-\epsilon)\leq 2\left(\frac{\pi}{4\sqrt{\mu}}-\epsilon\right)=\frac{\pi}{2\sqrt{\mu}}-2\epsilon,
\end{align}
where $\epsilon:=r_c-r$. 
Write the left-hand-sides of  \eqref{D-4-0} and \eqref{D-4-1} as
\begin{multline}
\label{eqn:ineq-Ixyz}
\g'(\dist(x, z)^2)\log_xz\cdot\log_xy+\g'(\dist(y, z)^2)\log_yz\cdot\log_yx = \\[5pt]
\dist(x, y)\underbrace{\left(
\g'(\dist(x, z)^2)\dist(x, z)\cos\angle(yxz)+\g'(\dist(y, z)^2)\dist(y, z)\cos\angle(xyz)
\right)}_{=:\mathcal{I}_{xyz}}.
\end{multline}
We will estimate $\mathcal{I}_{xyz}$ when $x\neq y$.
\smallskip

\noindent (Case 1: $\uc>0$) 
Write
\begin{align}\label{D-4}
\begin{aligned}
\mathcal{I}_{xyz}&=\left(
\frac{\dist(x, z)}{\sin\left(\sqrt{\uc}\dist(x, z)\right)}\g'(\dist(x, z)^2)
\right)\sin\left(\sqrt{\uc}\dist(x, z)\right) \cos\angle(yxz)\\
&\quad +\left(
\frac{\dist(y, z)}{\sin\left(\sqrt{\uc}\dist(y, z)\right)}\g'(\dist(y, z)^2)
\right)\sin\left(\sqrt{\uc}\dist(y, z)\right)\cos\angle(xyz).
\end{aligned}
\end{align}
By Lemma \ref{L2.3} we have the following inequalities:
\begin{align}\label{D-5}
\begin{aligned}
\cos\left(\sqrt{\uc}\dist(y, z)\right)\leq \cos\left(\sqrt{\uc}\dist(x, y)\right) \cos\left(\sqrt{\uc}\dist(x, z)\right)+\sin\left(\sqrt{\uc}\dist(x, y)\right) \sin\left(\sqrt{\uc}\dist(x, z)\right)\cos\angle(yxz),\\[5pt]
\cos\left(\sqrt{\uc}\dist(x, z)\right)\leq \cos\left(\sqrt{\uc}\dist(x, y)\right) \cos\left(\sqrt{\uc}\dist(y, z)\right)+\sin\left(\sqrt{\uc}\dist(x, y)\right) \sin\left(\sqrt{\uc}\dist(y, z)\right)\cos\angle(xyz).
\end{aligned}
\end{align}
We can rewrite \eqref{D-5} as
\begin{align}\label{D-6}
\begin{aligned}
\sin\left(\sqrt{\uc}\dist(x, z)\right)\cos\angle(yxz)\geq \frac{\cos\left(\sqrt{\uc}\dist(y, z)\right)-\cos\left(\sqrt{\uc}\dist(x, y)\right) \cos\left(\sqrt{\uc}\dist(x, z)\right)}{\sin\left(\sqrt{\uc}\dist(x, y)\right) },\\[2pt]
\sin\left(\sqrt{\uc}\dist(y, z)\right)\cos\angle(xyz)\geq \frac{\cos\left(\sqrt{\uc}\dist(x, z)\right)-\cos\left(\sqrt{\uc}\dist(x, y)\right) \cos\left(\sqrt{\uc}\dist(y, z)\right)}{\sin\left(\sqrt{\uc}\dist(x, y)\right) }.
\end{aligned}
\end{align}

Now substitute \eqref{D-6} into \eqref{D-4} to obtain
\begin{align}\label{D-7}
\begin{aligned}
\mathcal{I}_{xyz}&\geq \left(
\frac{\dist(x, z)}{\sin\left(\sqrt{\uc}\dist(x, z)\right)}\g'(\dist(x, z)^2)
\right) \frac{\cos\left(\sqrt{\uc}\dist(y, z)\right)-\cos\left(\sqrt{\uc}\dist(x, y)\right) \cos\left(\sqrt{\uc}\dist(x, z)\right)}{\sin\left(\sqrt{\uc}\dist(x, y)\right) }\\
&\quad +\left(
\frac{\dist(y, z)}{\sin\left(\sqrt{\uc}\dist(y, z)\right)}\g'(\dist(y, z)^2)
\right) \frac{\cos\left(\sqrt{\uc}\dist(x, z)\right)-\cos\left(\sqrt{\uc}\dist(x, y)\right) \cos\left(\sqrt{\uc}\dist(y, z)\right)}{\sin\left(\sqrt{\uc}\dist(x, y)\right) }.
\end{aligned}
\end{align}
We use the following simple identity:
\begin{align}\label{D-7-1}
A_1B_1+A_2B_2=\frac{1}{2}(A_1+A_2)(B_1+B_2)+\frac{1}{2}(A_1-A_2)(B_1-B_2)
\end{align}
to rewrite the right-hand-side of \eqref{D-7} as 
\begin{align*}
& \frac{1}{2}\left(
\frac{\dist(x, z)}{\sin\left(\sqrt{\uc}\dist(x, z)\right)}\g'(\dist(x, z)^2)+\frac{\dist(y, z)}{\sin\left(\sqrt{\uc}\dist(y, z)\right)}\g'(\dist(y, z)^2)
\right)\\
&\hspace{6cm}\times \frac{(\cos\left(\sqrt{\uc}\dist(x, z)\right)+\cos\left(\sqrt{\uc}\dist(y, z)\right))(1-\cos\left(\sqrt{\uc}\dist(x, y)\right))}{\sin\left(\sqrt{\uc}\dist(x, y)\right)}\\
& + \frac{1}{2}\left(
\frac{\dist(x, z)}{\sin\left(\sqrt{\uc}\dist(x, z)\right)}\g'(\dist(x, z)^2)-\frac{\dist(y, z)}{\sin\left(\sqrt{\uc}\dist(y, z)\right)}\g'(\dist(y, z)^2)
\right)\\
&\hspace{6cm}\times \frac{(-\cos\left(\sqrt{\uc}\dist(x, z)\right)+\cos\left(\sqrt{\uc}\dist(y, z)\right))(1+\cos\left(\sqrt{\uc}\dist(x, y)\right))}{\sin\left(\sqrt{\uc}\dist(x, y)\right)}\\
&=:\mathcal{I}_{xyz}^1+\mathcal{I}_{xyz}^2.
\end{align*}

Since we assumed that $\frac{\theta}{\sin(\sqrt{\uc}\theta)}\g'(\theta^2)$ is  non-decreasing and non-negative, we infer that 
\[
\frac{\dist(x, z)}{\sin\left(\sqrt{\uc}\dist(x, z)\right)}\g'(\dist(x, z)^2)-\frac{\dist(y, z)}{\sin\left(\sqrt{\uc}\dist(y, z)\right)}\g'(\dist(y, z)^2)
\]
and
\[
-\cos\left(\sqrt{\uc}\dist(x, z)\right)+\cos\left(\sqrt{\uc}\dist(y, z)\right)
\]
have the same sign. 
This yields
\[
\mathcal{I}_{xyz}^2\geq 0,
\]
and hence,
\begin{align}\label{D-8}
\mathcal{I}_{xyz}\geq \mathcal{I}_{xyz}^1+\mathcal{I}_{xyz}^2\geq \mathcal{I}_{xyz}^1.
\end{align}

By triangle inequality we have
\[
\dist(x, z)+\dist(y, z)\geq \dist(x, y),
\]
which implies that
\begin{align}\label{D-8-1}
\max\left\{
\dist(x, z), \dist(y, z)
\right\}\geq \frac{1}{2}\dist(x, y).
\end{align}
Since $\frac{\theta}{\sin(\sqrt{\uc}\theta)}g'(\theta^2)$ is non-decreasing and non-negative, we have
\begin{align}\label{D-8-1-1}
\frac{\dist(x, z)}{\sin\left(\sqrt{\uc}\dist(x, z)\right)}\g'(\dist(x, z)^2)+\frac{\dist(y, z)}{\sin\left(\sqrt{\uc}\dist(y, z)\right)}\g'(\dist(y, z)^2)\geq \frac{\dist(x, y)/2}{\sin(\sqrt{\uc}\dist(x, y)/2)}g'(\dist(x, y)^2/4).
\end{align}
Here, we dropped the smaller term, and we estimated the larger term using \eqref{D-8-1}. 

On the other hand, we have
\begin{equation}\label{D-8-2}
\begin{aligned}
\mathcal{I}_{xyz}^1&= \frac{1}{2}\left(
\frac{\dist(x, z)}{\sin\left(\sqrt{\uc}\dist(x, z)\right)}\g'(\dist(x, z)^2)+\frac{\dist(y, z)}{\sin\left(\sqrt{\uc}\dist(y, z)\right)}\g'(\dist(y, z)^2)
\right)\\[5pt]
&\hspace{4cm}\times(\cos\left(\sqrt{\uc}\dist(x, z)\right)+\cos\left(\sqrt{\uc}\dist(y, z)\right))\tan\left(\sqrt{\uc}\dist(x, y)/2\right).
\end{aligned}
\end{equation}

Note that $\cos(\sqrt{\mu}\dist(x, z))$ and $\cos(\sqrt{\mu}\dist(y, z))$ are non-negative by \eqref{D-4-3}. Then combine \eqref{D-8-1-1} and \eqref{D-8-2} to get
\begin{align}\label{D-8-3}
\begin{aligned}
\mathcal{I}_{xyz}^1&\geq\frac{1}{2}\frac{\dist(x, y)/2}{\sin(\sqrt{\uc}\dist(x, y)/2)}g'(\dist(x, y)^2/4)(\cos\left(\sqrt{\uc}\dist(x, z)\right)+\cos\left(\sqrt{\uc}\dist(y, z)\right))\tan\left(\sqrt{\uc}\dist(x, y)/2\right)\\[5pt]
&=\frac{1}{4}\frac{\dist(x, y)}{\cos(\sqrt{\uc}\dist(x, y)/2)}g'(\dist(x, y)^2/4)(\cos\left(\sqrt{\uc}\dist(x, z)\right)+\cos\left(\sqrt{\uc}\dist(y, z)\right)).
\end{aligned}
\end{align}

Next, we will estimate
\[
\cos(\sqrt{\uc}\dist(x, z))+\cos(\sqrt{\uc}\dist(y, z)).
\]
By \eqref{D-4-3}, $\dist(x, z) \leq \frac{\pi}{2\sqrt\uc}-2\epsilon$, which implies
\[
\cos(\sqrt{\uc}\dist(x, z)) \geq \cos\left(\sqrt{\uc}\left(\frac{\pi}{2\sqrt\uc}-2\epsilon\right)\right)=\sin(2\sqrt{\uc}\epsilon).
\]
The same inequality also holds for $\cos(\sqrt{\uc}\dist(y, z))$, as $\dist(y, z) \leq  \frac{\pi}{2\sqrt\uc}-2\epsilon$ as well. By substituting the inequality above into \eqref{D-8-3} we get
\begin{equation}\label{D-8-4}
\mathcal{I}_{xyz}^1\geq \frac{1}{2} \frac{\dist(x, y)}{\cos(\sqrt{\uc}\dist(x, y)/2)}g'(\dist(x, y)^2/4)\sin(2\sqrt{\uc}\epsilon)\geq \frac{\sin(2\sqrt{\uc}\epsilon)}{2}\dist(x, y)g'(\dist(x, y)^2/4).
\end{equation}

Finally, combine \eqref{eqn:ineq-Ixyz}, \eqref{D-8} and \eqref{D-8-4} to get
\begin{align*}
\g'(\dist(x, z)^2)\log_xz\cdot\log_xy+\g'(\dist(y, z)^2)\log_yz\cdot\log_yx \geq  \frac{\sin(2\sqrt{\uc}\epsilon)}{2}\dist(x, y)^2\g'(\dist(x, y)^2/4),
\end{align*}
which is the desired result.
\medskip

\noindent (Case 2: $\uc\leq 0$) Using again Lemma \ref{L2.3} we have for this case:
\begin{align*}
\dist(y, z)^2\geq \dist(x, y)^2+\dist(x, z)^2-2\dist(x, y)\dist(x, z)\cos\angle(yxz),\\[2pt]
\dist(x, z)^2\geq \dist(x, y)^2+\dist(y, z)^2-2\dist(x, y)\dist(y, z)\cos\angle(xyz).
\end{align*}
We rewrite the above as
\begin{align*}
\dist(x, z)\cos\angle(yxz)\geq\frac{\dist(x, y)^2+\dist(x, z)^2-\dist(y, z)^2}{2\dist(x, y)},\\[2pt]
\dist(y, z)\cos\angle(xyz)\geq\frac{\dist(x, y)^2+\dist(y, z)^2-\dist(x, z)^2}{2\dist(x, y)},
\end{align*}
and use it to estimate $\mathcal{I}_{xyz}$ as
\begin{align*}
\mathcal{I}_{xyz}
&\geq\g'(\dist(x, z)^2)\left(\frac{\dist(x, y)^2+\dist(x, z)^2-\dist(y, z)^2}{2\dist(x, y)}\right)+\g'(\dist(y, z)^2)\left(\frac{\dist(x, y)^2+\dist(y, z)^2-\dist(x, z)^2}{2\dist(x, y)}\right).
\end{align*}

We use again \eqref{D-7-1} to find
\begin{align*}
\mathcal{I}_{xyz}&\geq\frac{1}{2}\left(
\g'(\dist(x, z)^2)+\g'(\dist(y, z)^2)
\right)\left(
\frac{\dist(x, y)^2+\dist(x, z)^2-\dist(y, z)^2}{2\dist(x, y)}+\frac{\dist(x, y)^2+\dist(y, z)^2-\dist(x, z)^2}{2\dist(x, y)}
\right)\\[5pt]
&+
\frac{1}{2}\left(
\g'(\dist(x, z)^2)-\g'(\dist(y, z)^2)
\right)\left(
\frac{\dist(x, y)^2+\dist(x, z)^2-\dist(y, z)^2}{2\dist(x, y)}-\frac{\dist(x, y)^2+\dist(y, z)^2-\dist(x, z)^2}{2\dist(x, y)}
\right)\\[5pt]
&=\frac{\dist(x, y)}{2}\left(
\g'(\dist(x, z)^2)+\g'(\dist(y, z)^2)
\right)+\frac{1}{2\dist(x, y)}\left(
\g'(\dist(x, z)^2)-\g'(\dist(y, z)^2)
\right)(\dist(x, z)^2-\dist(y, z)^2).
\end{align*}

Since $\g'$ is non-decreasing, we have
\[
\frac{1}{2\dist(x, y)}\left(
\g'(\dist(x, z)^2)-\g'(\dist(y, z)^2)
\right)(\dist(x, z)^2-\dist(y, z)^2)\geq0.
\]
Finally, we get
\[
\mathcal{I}_{xyz}\geq\frac{\dist(x, y)}{2}\left(
\g'(\dist(x, z)^2)+\g'(\dist(y, z)^2)
\right)\geq \frac{\dist(x, y)}{2} g'(\dist(x, y)^2/4),
\]
where for the second inequality we used \eqref{D-8-1}. By combining \eqref{eqn:ineq-Ixyz} with the inequality above one can then reach the desired result.
\qed

\bibliographystyle{abbrv}
\def\url#1{}
\bibliography{lit.bib}

\end{document}